\documentclass[12pt]{amsart}
\usepackage{latexsym,amssymb,amsfonts,amsmath, amsthm, amscd, euscript, amsfonts, MnSymbol}
\usepackage[english]{babel}
\usepackage{enumerate}
\usepackage{verbatim}
\usepackage{multicol}
\usepackage{graphicx}
\usepackage{wrapfig}
\usepackage{pdfsync}
\usepackage{epsfig}
\usepackage{epstopdf}
\usepackage{color}
\usepackage{enumerate}
\usepackage{verbatim}
\usepackage{cases}
\usepackage{lineno}

\usepackage{geometry}
\newtheorem{theorem}{Theorem}
\newtheorem{claim}{Claim}

\newtheorem{problem}{Problem}

\newtheorem{lemma}[theorem]{Lemma}

\newtheorem{proposition}[theorem]{Proposition}
\newtheorem{comments}[theorem]{Comments}
\newtheorem{corollary}[theorem]{Corollary}

\newtheorem{definition}{Definition}

\newcommand{\QQ}{{\mathbb Q}}
\newcommand{\NN}{{\mathbb N}}
\newcommand{\ZZ}{{\mathbb Z}}
\newcommand{\RR}{{\mathbb R}}

\DeclareMathOperator{\ainc}{Inc}

\DeclareMathOperator{\Conv}{Conv}

\DeclareMathOperator{\comp}{{Comp}}
\DeclareMathOperator{\id}{{Id}}
\DeclareMathOperator{\prim}{Prim}

\DeclareMathOperator{\age}{Age}

\DeclareMathOperator{\Inc}{\overline{Inc}}
\DeclareMathOperator{\Id}{\mathbf{Id}}
\DeclareMathOperator{\dual}{dual}
\DeclareMathOperator{\forb}{Forb}

\title{Hereditary classes of ordered sets of width at most two}
\author [Maurice Pouzet]{Maurice Pouzet}
\address{Univ. Lyon, Universit\'e Claude-Bernard Lyon1, CNRS UMR 5208, Institut Camille Jordan, 43, Bd. du 11 Novembre 1918, 69622
Villeurbanne, France et Department of Mathematics and Statistics, University of Calgary, Calgary, Alberta, Canada}
\email{pouzet@univ-lyon1.fr }

\author[Imed Zaguia]{Imed Zaguia*}\thanks{*Corresponding author. Supported by the Canadian Defence Academy Research Program, NSERC and LABEX MILYON (ANR-10-LABX-0070) of Universit\'e de Lyon within the program ''Investissements d'Avenir (ANR-11-IDEX-0007'' operated by the French National Research Agency (ANR)}
\address{Department of Mathematics \& Computer Science, Royal Military College of Canada,
P.O.Box 17000, Station Forces, Kingston, Ontario, Canada K7K 7B4}
\email{zaguia@rmc.ca}

\date{\today}


\begin{document}

\maketitle

\dedicatory{$\dag$\;  Dedicated to the memory of Pierre Rosenstiehl.}

\begin{abstract}This paper is a contribution to the study of hereditary classes of relational structures, these classes being quasi-ordered by embeddability. It deals with the specific case of ordered sets  of width two and the corresponding bichains and incomparability graphs.

Several open problems about hereditary classes of relational structures which have been considered over the years have positive answer in this case. For example, well-quasi-ordered hereditary classes of finite  bipartite permutation graphs, respectively finite 321-avoiding permutations, have been characterized by Korpelainen, Lozin and Mayhill, respectively by  Albert, Brignall, Ru\v{s}kuc and Vatter.

In this paper we present an overview of properties of these hereditary classes in the framework of the Theory of Relations as presented by Roland Fra\"{\i}ss\'e.

We provide another proof of the results mentioned above. It is based on  the existence of a countable universal poset of width two, obtained by the first author in 1978, his notion of multichainability (1978) (a kind of analog to letter-graphs), and metric properties of incomparability graphs. Using Laver's theorem (1971) on better-quasi-ordering (bqo) of countable chains we prove that a wqo hereditary class of finite or  countable bipartite permutation graphs is necessarily bqo. This gives a positive answer to a conjecture of Nash-Williams (1965) in this case. We extend a previous result of Albert et al. by proving that if a hereditary class of finite, respectively countable, bipartite permutation graphs is wqo, respectively bqo, then the corresponding hereditary classes of posets of width at most two and bichains are wqo, respectively bqo.

Several notions of labelled wqo are also considered. We prove that they are all equivalent in the case of bipartite permutation graphs, posets of width at most two and the corresponding bichains. We characterize hereditary classes of  finite bipartite permutation graphs which remain wqo when labels from a wqo are added. Hereditary classes of  posets of width two, bipartite permutation graphs and the corresponding bichains having finitely many bounds are also characterized.

We prove that a hereditary class of finite bipartite permutation graphs is not wqo if and only if it embeds the poset of finite subsets of $\NN$ ordered by set inclusion. This answers a long standing conjecture of the first author in the case of bipartite permutation graphs.
\end{abstract}


\section{Introduction}

This paper is a contribution to the study of hereditary classes of relational structures, these classes being quasi-ordered by embeddability. It is concerned with hereditary classes of posets of width at most two and hereditary classes of  derived structures, like their incomparability graphs and their associated bichains. Several open problems about hereditary classes of relational structures  have been considered since the seventies. By lack of significant results, these problems were not presented in a coherent way. The situation is changing. As we will illustrate,  those problems have a positive answer in the case of posets of width at most two and their derived structures.

In  this  paper we introduce the framework of the theory of relations as developed by Roland Fra\"{\i}ss\'e and subsequent investigators. We present first some of the central questions about well-quasi-ordered hereditary classes (we have chosen five, but there are many other). Next, we present known  and new results about bipartite permutation graphs and the corresponding bichains. And we conclude by giving  a positive answer to  the questions mentioned.

At the core  of the framework of the theory of relations  is the notion of embeddability, a quasi-order between relational structures. We recall that a relational structure $R$ is \emph{embeddable} in a relational structure $R'$,  and we set $R\leq R'$, if $R$ is isomorphic to an induced substructure of $R'$. Several  important notions in the study of these structures, like hereditary classes, ages, bounds, derive from this quasi-order.

A class $\mathcal C$ of relational structures is \emph{hereditary} if it contains every relational structure that embeds into a member of $\mathcal C$. The \emph{age} of a relational structure $R$ is the class $\age(R)$  of all finite relational structures, considered up to isomorphy, which embed into $R$. This is an \emph{ideal}, that is  a  nonempty,  hereditary  and  \emph{up-directed} class $\mathcal{C}$ (any pair of members of $\mathcal C$ are embeddable in some element of $\mathcal C$). As shown by Fra\"{\i}ss\'e (see chapter 10 of  \cite{fraissetr}), in the case of relational structures made of finitely many relations, the converse holds: ideals are ages. Among hereditary classes there are those that are well-quasi-ordered. A class $\mathcal{C}$ is \emph{well-quasi-ordered} by embeddability, wqo for short, if every infinite sequence of members of $\mathcal{C}$ contains an infinite subsequence which is increasing with respect to embeddability.

A \emph{bound} of a hereditary class $\mathcal C$ of finite relational structures (e.g. graphs, ordered sets, bichains) is any relational structure $R\not \in \mathcal C$ such that every proper induced substructure of $R$ belongs to $\mathcal C$. The bounds of $\mathcal C$ form an antichain. Note that
if $\mathcal C'$ is a proper subclass of $\mathcal C$, then every bound of $\mathcal C'$ not in $\mathcal C$ is also a bound of $\mathcal C$. Furthermore, if $\mathcal C$ is wqo, then only finitely many bounds of $\mathcal C'$ are in $\mathcal C$. 

For a wealth of information on these notions see Fra\"{\i}ss\'e's book \cite{fraissetr}.\\

Several investigations have been devoted to hereditary classes, ages and their bounds, particularly those classes that are well-quasi-ordered. This is particulary done in the case of structures like graphs, posets and notably in the case of permutation graphs and bichains (see some surveys like \cite{pouzetw.q.o.bqo, pouzet2006}, \cite{klazar, vatter} and hopefully the forthcoming \cite{pouzet2022}). Several questions remain unsolved. For instance,

\begin{problem}
An age $\mathcal{C}$ of finite graphs is not well-quasi-ordered if and only if $\mathcal{C}$ embeds the poset of finite subsets of $\NN$ ordered by set inclusion.
\end{problem}

When true, the condition above is equivalent to other significant conditions (e.g. $\mathcal{C}$ contains uncountably many subages), see Proposition \ref{prop:well-founded}.

A strengthening of the notion of  well-quasi-order, \emph{better-quasi-order} (bqo), was introduced in 1965 by Nash-Williams \cite{nashwilliams1}. Nash-Williams \cite{nashwilliams1} p.700,  asserted  that "one  is inclined to  conjecture  that  most  wqo sets which  arise in a reasonably  'natural'  manner  are likely to be bqo"

\begin{problem}Is a hereditary class of finite structures with finite signature  bqo whenever it is wqo?
\end{problem}

It is not known if the answer is positive for hereditary classes of finite graphs  (it is also unsolved for the class of finite graphs equipped with the minor quasi-order  while, by Robertson-Seymour's theorem, this class is wqo). We prove that the answer is positive for the class of finite bipartite permutation graphs and the corresponding classes of bichains and posets (see Theorem \ref{thm:2}).  However, it should be noted that the answer is negative if the signature is infinite and the arity is unbounded. Also, there are wqo ages which are not bqo. Examples are based on a construction given in \cite{pouzetsobranisa}.

In general, infinitary constructions do not preserve well-quasi-order. An example due to Rado \cite{rado} illustrates this. The notion of bqo was invented by Nash-Williams to prove that some classes of countable structures are wqo (see Nash-Williams \cite{nashwilliams1} for countable trees under the topological minor relation and Laver (1971) \cite{laver} for countable chains under embedding),  so one may ask why it is important to know whether wqo ages which are made of finite structures are in fact bqo.  The reason is twofold. First, in several instances, the complexity of an infinite  hereditary class $\mathcal{C}$ of finite structures can be better represented by a single infinite relational structure $R$ (this requires that $\mathcal{C}$  be an age and $\mathcal{C}= \age(R)$). This is one of the originalities of the approach of combinatorial problems by the theory of relations. Next, properties of families of ages can be reflected in terms of properties of families of countable structures. For example, a countable antichain of ages yields trivially a countable antichain of countable structures  (the same property holds with chains, see Theorem V-2.3 in \cite{pouzet79} and also \cite{pouzetsobranisa}). Now, if an age embeds Rado's example \cite{rado} (this is an example of a poset that is wqo but not bqo), it will contain an infinite antichain of ages. Up to now, the known examples of ages embedding this poset and which are wqo are made of relational structures with an unbounded signature.


In order to prove that some classes of structures are wqo or bqo, one is lead to prove more: these classes remain wqo under the embeddability quasi-order when the structures are labelled with the elements of a quasi-order. Precisely, let $\mathcal{C}$ be a class of relational structures, e.g.,  graphs, posets, etc., and $Q$ be  a quasi-ordered set or a poset. If $R\in \mathcal C$, a \emph{labelling of $R$ by $Q$} is any map $f$ from  the domain of $R$ into $Q$. Let   $\mathcal{C}\cdot Q$ denotes the collection of $(R,f)$ where $R\in \mathcal{C}$ and $f: R\rightarrow Q$ is  a labelling. This class is quasi-ordered  by $(R,f)\leq (R',f')$ if there exists an embedding $h: R\rightarrow R'$ such that $f(x)\leq (f'\circ h)(x)$ for all $x\in R$.  The class $\mathcal{C}$ is \emph{hereditary wqo} if $\mathcal{C}\cdot Q$ is wqo for every wqo $Q$. The class $\mathcal{C}$ is $n^{-}$-wqo if the class $\mathcal C_{n^-}$ of $(R,a_1,\ldots,a_n)$ where $R\in \mathcal {C}$ and $a_1,\ldots,a_n\in R$ is wqo. Analogous notions with bqo replacing wqo can be also defined.

These notions were introduced by the first author (see \cite{pouzetthesis} and \cite{oudrar-pouzet2016}). We do not know if they are different. 

\begin{problem}A hereditary class of finite structures is either hereditary wqo  or it is not $2^-$-wqo.
\end{problem}

\begin{problem} If a hereditary class $\mathcal C$ of finite structures is hereditary bqo, is it true that the class of countable structures $R$ such that $\age(R) \subseteq \mathcal C$ bqo?
\end{problem}

A huge amount of work has been devoted to enumeration problems. In the framework of the theory of relations, the \emph{profile} of a hereditary class $\mathcal C$ of finite relational structures is the function $\varphi_{\mathcal C}$ which counts for every nonnegative integer $n$ the number $\varphi_{\mathcal C}(n)$ of members of $\mathcal C$ on  $n$ elements counted up to isomorphism. Solving a question of  Fra\"{\i}ss\'e and him, the first author shown in 1971 that the profile of an infinite age is nondecreasing. The proof, based on Ramsey theorem,  appears in Fra\"{\i}ss\'e's Cours de Logique  1971 (Exercice 8, p.113, \cite{fraisse-courslogique1}). Then, jumps in the growth rate of  profiles have been observed in 1978 \cite{pouzetthesis, pouzetmontreal}, later, hereditary classes  whose the profile is bounded by a polynomial  have been studied \cite{pouzet-thiery1, pouzet-thiery2};  the paper \cite{pouzet2006} contains  an overview, with a particular emphasis on profiles whose generating function is a rational fraction, the Cameron approach,  and some problems linking properties of profiles and  of rational languages (e.g., Chomsky-Sch\"{u}tzenberger’s theorem)."

\begin{problem} \label{problem:pouz-sob}  Is the profile of a wqo  hereditary class $\mathcal C$  of finite relational structures with a bounded signature majorized by some exponential?
\end{problem}

 Some hypothesis (like the boundedness of the signature) is needed but the answer is not known in the case of graphs. Without this hypothesis, one can built examples of wqo hereditary classes of arbitrarily large profiles (a joint result  not yet published \cite{pouzet-sobrani}).  The  problem above with the  hypothesis that $\mathcal C$ is hereditary wqo was proposed in 2003 \cite {pouzetsobranisp}.

During the last fifteen years, many results concerning the enumeration of  classes of binary relations have appeared. Among them, it is worth noting the results of Albert and Atkinson \cite{A-A}, Balog et al,\cite{B-B-S-S, B-B-M06, B-B-M06/2, BBM07} (see \cite{klazar, vatter} for an overview). These results as well as the  recent result by Albert et al \cite {albert-brignall-Ruskuc-vatter} 2019 proving that the generating series of the profile of a hereditary class of bichains is a rational fraction support a positive answer to the problem above.

There are other problems related to those; these are mentioned in the text, with label (e.g., Problem \ref{problem6}) or not.

\section{Known results}
In this section we present known results about bipartite permutation graphs and the corresponding bichains.

Recall that a finite graph is  a  \emph{permutation graph} if  its vertices represent the elements of a permutation, and  its edges represent pairs of elements that are reversed by the permutation. 

The class of finite permutation graphs is hereditary and  not well-quasi-ordered  \cite{bona}.  Indeed,  the set of \emph{double-ended forks} form an infinite antichain among bipartite permutation graphs. (see Figure \ref{fig:doublefork}).

A \emph{double-ended fork  $\mathrm{DF}_{n}$ of length $n$}, for $n\geq 1$, is the graph obtained from a path $\mathrm P_n$ on $n$ vertices by adding two vertices adjacent to each end vertex of  $\mathrm P_n$,  these  two vertices being  nonadjacent.

\begin{figure}[h]
\begin{center}
\leavevmode \epsfxsize=3in \epsfbox{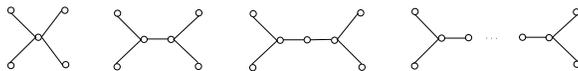}
\end{center}
\caption{Double-ended forks} \label{fig:doublefork}
\end{figure}

Korpelainen and Lozin   proved that the class of finite bipartite permutation graphs not containing a path of length $k$ is wqo (see Theorem 26 in \cite{korpelainen-lozin}).  More generally, Lozin and Mayhill \cite{lozin-mayhill} proved the following (see Theorem 7 page 520),

\begin{theorem}\label{thm:korpelainen-lozin}
A hereditary class  $\mathcal C$ of finite bipartite permutation graphs is wqo if and only if there is  some integer $n$ such that $\mathcal C$ contains no double-ended fork of length greater than  $n$.
\end{theorem}

These results are based on a result of Lozin and Gabor (Theorem 6  in \cite{lozin-gabor}) on the existence of  universal bipartite permutation graphs  and the  notion  of \emph{letter graphs} introduced by Petkov\u{s}ek \cite{petkovsek}.

Finite bipartite permutation graphs are incomparability graphs of ordered sets of width at most two (that is antichains have cardinality at most two). If a poset $P:= (V, \leq )$ has width two,  then the order $\leq $ is the intersection of two linear orders $\leq_1$ and $\leq_2$ on $V$ (this follows from Dilworth's Theorem \cite{dilworth} that the dimension of a poset is at most its width). Let us call \emph {bichain} every relational  structure $B:= (V, (\leq_1, \leq_2))$  made of  a set $V$ and two  linear orders $\leq_1$ and $\leq_2$; denote by $o( B)$ the poset $(V, \leq_1\cap \leq_2)$. Then each of our posets is of the form $o(B)$ for some bichain. Finite bichains are a way of looking at (finite) permutations. In fact, each  bichain $B$ on a $n$-element set  can be coded in a unique way by a permutation $\sigma$ of $\{1, \dots, n\}$ (as a $n$-element chain can be coded by the natural order on $\{1, \dots n\}$). As observed by Cameron \cite{cameron}, the embeddability between bichains corresponds to an order $\leq$ on the set $\Omega$ of permutations of $\{0, \dots, n-1\}$, for  $n\in \NN$. Given permutations $\pi$ and $\sigma$, we say that $\pi$ \emph{contains} $\sigma$, and write $\sigma \leq \pi$, if $\pi$ has a subsequence $\pi(i_1)\ldots \pi (i_{|\sigma|})$ of the same length as $\sigma$ that is order isomorphic to $\sigma$ (i.e., $\pi (i_s) < \pi (i_t )$ if and only if $\sigma(s) < \sigma (t)$ for all $1 \leq s, t \leq |\sigma|)$; otherwise, we say that $\pi$ \emph{avoids} $\sigma$.

Denote by $Av(321)$ the set of 321-avoiding permutations. Then finite bichains associated to posets of width at most two correspond to $321$-avoiding permutations.

Motivated by the resolution of the Stanley-Wilf Conjecture  (by Marcus and  Tard\"os \cite{marcus-tardos} 2004), hereditary classes of permutations (defined as  classes $\mathcal C$ of permutations $\sigma$ such that  $\sigma \leq \tau\in \mathcal C $ implies  $\sigma \in \mathcal C$) and their generating series have been intensively studied  (see \cite{vatter} for an overview). Due to  Cameron's observation, their study fits  within the study of  hereditary  classes of finite  relational structures.

Let $U$ denote the set of all permutations $\pi$ for which the corresponding permutation graph is isomorphic to a double-ended fork. The following result is due to Albert et al. (see Theorem 9.3. in \cite{albert-brignall-Ruskuc-vatter}).

\begin{theorem}\label{thm:albert-brignall-Ruskuc-vatter} A subclass $\mathcal{C}\subseteq Av(321)$ is well-quasi-ordered if and only if $\mathcal{C}\cap U$ is finite.
\end{theorem}

Although we do not deal with generating functions in this paper, the following result, due to Albert et al., (see Theorem 1.1. \cite{albert-brignall-Ruskuc-vatter}) is worth mentioning.

\begin{theorem}
If a proper subclass of the $321$-avoiding permutations has finitely many bounds or is well-quasi-ordered, then it has a rational generating function.
\end{theorem}

This result solves positively Problem \ref{problem:pouz-sob} in the case of bichains whose the intersection order has width two.

The following result is due to Brignall and Vatter (see Theorem 7.17 in \cite{brignall-vatter}). It expresses the fact that hereditary wqo does translate up from permutation graphs to the corresponding bichains.

\begin{theorem}
The permutation class $\mathcal{C}$ is hereditary w.q.o if and only if the corresponding class of permutation graphs is hereditary w.q.o.
\end{theorem}

From bichains to the corresponding posets and going down to incomparability graphs, there is an order-preserving mapping. In particular, this
means that if we have a wqo/bqo set of objects at one level, then the image of that set on the level down will also be wqo/bqo. We only know that this works in the downward direction. Thus the question of whether we can translate wqo (or bqo)  properties up is open, in general.

It follows from Albert et al. \cite{albert-brignall-Ruskuc-vatter} that this is the case for finite bipartite permutation graphs. We extend this result to countable incomparability graphs of posets of width at most two. We should mention that these are not necessarily permutation graphs. 

\section{Presentation of the results}
 Looking at hereditary classes of  posets of width at most two which are finite or countable, we  describe those  which are wqo and prove that in fact they are bqo (see Theorem \ref{thm:2} below),  a far reaching strengthening of the notion of wqo by Nash-Williams \cite{nashwilliams1}.

To a  class $\mathcal B$ of bichains  we may associate the class of posets $o\langle\mathcal B\rangle:=\{ o(B): B\in \mathcal B\}$ and to the class $o\langle\mathcal B\rangle$ we may associate the class of incomparability graphs $\ainc\langle o\langle\mathcal B\rangle\rangle:= \{\ainc(P): P  \in o\langle\mathcal B\rangle\}$. It should be noted that to an element of $\ainc\langle o\langle\mathcal B\rangle\rangle$ may correspond several elements of $o\langle\mathcal B\rangle$. Similarly, to an element of $o\langle\mathcal B\rangle$ may correspond several elements of $\mathcal B$.

A natural question then arises: how does the embeddability between elements of a given class translate to the embeddability between the corresponding elements in the other two classes? Trivially, an embedding between two bichains, up to a transposition, yields an embedding between the corresponding posets. An embedding between two posets, up to duality, yields an embedding between the corresponding incomparability graphs. It is not  known if the reverse of these implications holds in general. We have only few elements of an answer (see Theorems \ref{thm:embed-poset-graph} and \ref{thm:embed-bichain-poset}).

The embeddability considerations are important when studying the transferability of the wqo/bqo character between  the above mentioned three classes.





It is easy to see that  if $\mathcal B$ is a wqo/bqo class of bichains then, the class  $o\langle\mathcal B\rangle$ is wqo/bqo. Also,  if $o\langle\mathcal B\rangle$ is wqo/bqo then $\ainc \langle o\langle\mathcal B\rangle \rangle$ is wqo/bqo too. It is not  known if the reverse of these implications holds. That is,  if $\mathcal C$ is a wqo/bqo class of posets, for which the order is intersection of two linear orders, is the class $o^{-1}\langle \mathcal C \rangle: = \{ B: o(B)\in \mathcal C\}$ wqo/bqo? And if  $\ainc\langle\mathcal C \rangle$ is wqo/bqo, is $\mathcal C$ wqo/bqo? We prove that the answer to these two questions is yes  if $\mathcal C$ is a hereditary class of finite or countable posets of width at most two. As stated in Theorem \ref{thm:bqo-poset-bichain} below, parts of these implications  hold in a more general situation.

With the notions of lexicographical sum given in subsection \ref{subsection:lex} and  the notion of module given in section \ref{section:bichains}, we give the following definition.

\begin{definition}Let $\mathcal{D}$ be a class of posets that are finite or  countable, whose order is the intersection of two linear orders,  whose only modules are totally ordered, and whose incomparability graph is connected. Let $\widehat{\mathcal{D}}:= \sum \mathcal D$ be the collection of lexicographical sums, over finite or countable chains, of posets belonging to $\mathcal D$. Let ${\widehat{\mathcal{D}}}_{<\omega}$ be the subclass of $\widehat{\mathcal{D}}$ made of finite posets (if any).
\end{definition}

\begin{theorem} \label{thm:bqo-poset-bichain}   The following propositions are true.
\begin{enumerate}[$(1)$]
\item $\widehat{\mathcal{D}}$ is bqo if and only if $o^{-1}\langle \widehat{\mathcal{D}}\rangle $ is bqo;
\item ${\widehat{\mathcal{D}}}_{<\omega}$ is wqo (resp., bqo)  if and only if $\ainc \langle {\widehat{\mathcal{D}}}_{<\omega}\rangle$ is wqo (resp., bqo).
\end{enumerate}
\end{theorem}

The proof of Theorem \ref{thm:bqo-poset-bichain} is given in Subsection \ref{subsection:proofthm1}.

\begin{corollary}The three classes ${\widehat{\mathcal{D}}}_{<\omega}$,  $o^{-1}\langle \widehat{\mathcal{D}}\rangle $ and $\ainc \langle {\widehat{\mathcal{D}}}_{<\omega}\rangle$ are  wqo (resp., bqo)  if and only if one of these is wqo (resp., bqo).
\end{corollary}




Some observations are in order. Firstly  note that,  in the case of  infinite structures, the appearance of bqo is unavoidable. Indeed,  the class $\mathcal {S}_{\leq\omega}$ of countable chains is  $\sum {\mathcal D}$ where $\mathcal D$ consists of  the one-element poset, hence  the hypotheses of the theorem are fulfilled. The class of incomparability graphs of chains  consists of graphs without edges so is totally ordered  by the cardinality of their vertex-set,  hence is wqo. The fact  that the class of countable chains  is wqo,  a famous conjecture of Fra\"{\i}ss\'e,  was proved  by Laver (1971) \cite{laver} by means of the theory of bqo. Secondly,  the result above does not state that the class $\mathcal C$ is bqo whenever it is wqo. We do not know if this is true in general. Our third observation is that  the class $\mathcal C$ is not necessarily hereditary. Since  the condition on modules imposes that  the three-element antichain does not belong to $\mathcal C$ we infer that if $\mathcal C$ is hereditary, then $\mathcal C$  is made of posets of width at most two. In this case, due to our knowledge of these posets, we prove that $\mathcal C$ is bqo if and only if $\mathcal C$ is wqo (see the equivalence $(i) \Leftrightarrow (iii)$ in Theorem \ref{thm:2}). 

The  key tools in our study are a previous result of the first author \cite{pouzet78} (1978) on  the existence of a countable universal poset of width two, his notion of multichainability \cite{pouzetthesis} 1978 (a kind of analog to letter-graphs\footnote{A referee of this paper indicates that there is an analogous notion for permutations called \emph{geometric grid class} \cite {aleculozin, alecu-al}.}), metric properties of incomparability graphs (see \cite{pouzet-zaguia21}) and Laver's theorem on countable chains  \cite{laver}.

According to Th\'eor\`eme II.4.1 of  \cite{pouzet78}, every countable poset of width two embeds into the lexicographical sum of copies of a particular poset of width two, denoted by $\QQ_{\pi,2}$,  indexed by the chain of rational numbers $\QQ$.  The poset $\QQ_{\pi,2}$ is the set of couples $(x,i)$ with $x\in \QQ$ and $i\in\{0,1\}$ ordered by: $(x, i)\leq (y,j)$ if either $i=j$  and $x \leq y$ in $\QQ$ or $i \neq j$ and $x +\pi \leq y$ in $\QQ$. From this result  follows that every  countable poset $P$ of width at most two  whose  incomparability graph is connected embeds into  $\QQ_{\pi,2}$. Instead of $\QQ_{\pi,2}$ we consider the following uncountable poset.

Let $D:=[0,2\pi [\times \ZZ$ ordered by:
 \[(r,n)\leq (r',n') \mbox{ if } r\leq  r' \mbox{ in}\;   [0,2\pi[ \;  \mbox {and}\;  n\leq n' \mbox{ or } r\geq  r' \mbox{ and } n'\geq n+2.\]
For a positive integer $m\geq 1$,  let $D(m):=([0,2\pi[ \times [-m,m],\leq)$ where $[-m, m]\subset \ZZ$.

The poset  $D(m)$ has a special structure. It is multichainable (Lemma \ref{dm-multichainable}). A poset $P$ is \emph{multichainable} if there
is an enumeration of the elements $(v_{xy})_{(x,y)\in L\times V}$ of the vertices of $P$ where $V$ is finite, $L$ is totally ordered
such that for every local isomorphism $f$ of $L$, the map $(f,\id)$, where $\id$ is the identity map on $V$,  is a local isomorphism of
$P$. This notion was introduced in 1978 \cite{pouzetthesis} (and published in \cite {pouzetrm} subsection IV.3.3.2, p.341, and subsequent papers, e.g.,  \cite{pouzet2006}). It allows us to point out hereditary classes  which are \emph{hereditary wqo}, resp., \emph{hereditary bqo} (a class $\mathcal{C}$ has this property if the class $\mathcal{C}\cdot Q$ of structures belonging to  $\mathcal{C}$ and labelled by a wqo, resp., a bqo $Q$,  remains wqo, resp., bqo). If $R$ is a  multichainable structure, in particular a multichainable poset, then, with the help of Laver's theorem on countable chains, we get that the collection of countable structures whose age is included in $\age (R)$ is hereditary bqo (Proposition \ref{multichainablebqo}).  This fact applies to $D(m)$ (Corollary \ref{multichainablebqo2}).

Let $\mathrm P_k$ be the path on $k$ vertices. A class $\mathcal C$ of graphs is \emph{$\mathrm P_k$-free} if it does not contain $\mathrm P_k$. If $\mathcal C$ is hereditary this amounts to stating  that no member of $\mathcal C$ embeds $\mathrm P_k$. Hence, this class is  $\mathrm P_{k'}$-free for every $k'\geq k$.

Here is our first characterisation.

 \begin{theorem}\label{thm:1} Let $\mathcal{C}$ be a hereditary class of finite or countable posets  having width at most two. Then the following
 propositions are equivalent.
 \begin{enumerate}[(i)]
   \item The class $\ainc\langle \mathcal{C}\rangle$ of incomparability graphs of members of $\mathcal C$ is $P_{k}$-free for some integer~$k$;
   \item The class $\mathcal{C}_{<\omega}$ of finite members of $\mathcal C$ is  a subset of  $\age (D(m))$ for some integer $m\geq 1$;
  \item The class $\mathcal{C}$ is hereditary bqo;
  \item The class consisting of the elements of $\mathcal{C}_{<\omega}$ labelled by two constants is wqo.

 \end{enumerate}
 \end{theorem}

The proof of Theorem \ref{thm:1} is given in Section \ref{section:proofthm:1}.

If we replace  $\mathcal{C}$ by the associated class of bichains or by the associated class of incomparability graphs, then one gets a similar characterisation.

An immediate consequence of Theorem \ref{thm:1} is Corollary 26 of Korpelainen and Lozin \cite{korpelainen-lozin}.

\begin{corollary}\label{korpelainen-lozin}For any fixed $k$, the class of finite $\mathrm P_k$-free bipartite permutation graphs is wqo by the induced subgraph relation.
\end{corollary}

The fact that the class of $\mathrm P_k$-free bipartite permutation graphs remains wqo when its members are labelled with two constants is related to the notion of bound (a \emph{bound} of a hereditary class $\mathcal{C}$ of graphs is any  finite structure $Q$ not in $\mathcal{C}$ which is minimal w.r.t. embeddability, hence, $Q$ is a bound  of $\mathcal{C}$ if $Q\not \in \mathcal{C}$ while $Q \setminus \{x\}\in \mathcal{C}$ for every $x\in V(Q)$). An other consequence of Theorem \ref{thm:1} is Corollary \ref{cor:bounds}.

\begin{corollary}\label{cor:bounds}Let $\mathcal{C}$ be a hereditary class of finite bipartite permutation graphs. The following propositions are equivalent.
\begin{enumerate}[(i)]
\item The class $\mathcal{C}$ has finitely many bounds;
\item The class $\mathcal{C}$ is hereditary wqo;
\item The class $\mathcal {C}$ is $\mathrm P_k$-free for some nonnegative integer $k$.
\end{enumerate}
\end{corollary}
\begin{proof}
$(i) \Rightarrow (iii)$ If $\mathrm P_k\in \mathcal C$ with $k$ even then the cycle on $k+1$ vertices is a bound of $\mathcal C$. Indeed,  since $\mathcal C$  is made of bipartite graphs it contains no odd cycles. \\
$(iii)\Rightarrow (i)$ Follows from  implication $(i)\Rightarrow (iii)$ of Theorem \ref{thm:1} and the fact that a bqo is a  wqo.\\
 $(ii) \Rightarrow (i)$ Suppose that $\mathcal {C}$ has infinitely many bounds.  Let $(B_n)_{n\in \NN}$ be an enumeration of
the bounds of $\mathcal{C}$. For each $n\in \NN$ let $x_n\in B_n$ and $B'_n=B_n\setminus \{x_n\}$. The vertex $x_n$ leads to a
partition of the vertices of $B'_n$ into two sets: those who are adjacent to $x_n$, which we color green, and those who are not
adjacent to $x_n$, which we color red. Denote by $B'_n\cdot \{g,r\}$ the bi-colored graph $B_n$. The structures $(B'_n\cdot
\{g,r\})_{n\in \NN}$ are in $\mathcal{C}\cdot \{g,r\}$. Since $\mathcal{C}$ is hereditary wqo it follows that there are integers
$i<j$ such that there is a color preserving embedding from $B'_i\cdot \{g,r\}$ into $B'_j\cdot \{g,r\}$. This embedding clearly yields
an embedding of $B_i$ into $B_j$. This is impossible since the bounds of $\mathcal C$ are pairwise incomparable.  Hence,
$\mathcal{C}$ has finitely many bounds as required.
\end{proof}

Implication $(ii)\Rightarrow (i)$ holds for any hereditary class of finite relational structures made of a fixed finite numbers of relations (not necessarily binary)  \cite{pouzet72}. On an other  hand, implication $(i)\Rightarrow (ii)$ may be  false if $\mathcal C$ is made of bipartite graphs or  of  permutation graphs (not necessarily bipartite).   Indeed, in \cite{korpelainen-lozin}, Korpelainen and Lozin proved that the class of $\mathrm P_7$-free bipartite graphs is not wqo, hence not hereditarily wqo. While  Brignall    and Vatter \cite{brignall-engen-vatter} have shown that some hereditary classes of permutations graphs are wqo with finitely many bounds but are not hereditary wqo,  answering negatively a question of    Korpelainen, Lozin, and Razgon \cite{korpelainen-lozin-razgon}). Implication $(i)\Rightarrow (iii)$ holds for other classes of graphs. Namely, any hereditary class of comparability graphs or incomparability graphs. Indeed, odd cycles are neither comparability nor incomparability graphs. Implication $(iii)\Rightarrow (i)$ is false in general,  even in the case of incomparability graphs. Indeed, let  $X$ be an infinite and coinfinite subset of $\NN$,  let $H_X$ be the direct sum $\bigoplus_{n\in X} C_{2n}$   of cycles $C_{2n}$  and let $\mathcal C$ be the age of the complement of $H_X$. Since  $\mathcal C$ consists of complements of finite bipartite graphs, this is a $\mathrm P_{5}$-free class of incomparability graphs. It has infinitely many bounds, since for every $n\not \in X$, the complement of $C_{2n}$ is a bound of $\mathcal C$.

Despite the fact that  posets of width at most two share several  properties with  bipartite permutation graphs,  if $\mathcal C$ consists of  finite posets of width two, implication $(i)\Rightarrow (ii)$ may be  false.  Indeed, let  $P_2$ be the poset represented Figure \ref{Fig3} whose incomparability graph is the infinite path $\mathrm P_{\infty}$. Then $\age(P_2)$  has finitely many bounds, its age is wqo but is not hereditarily wqo. A characterisation of ages of posets of width two which are wqo and have only finitely many bounds is given in Corollary \ref{cor-wqo-finite} as a consequence of the  characterisation of wqo hereditary classes of posets of width two. This  characterisation uses  double-ended forks.

\begin{definition}For $n\in \NN$, let $\forb(n)$ be the class of finite or countable posets having width at most two such that no $\mathrm{DF}_{n'}$, for $n'\geq n$, belongs to $\ainc\langle \forb(n) \rangle$. Let $\forb_{<\omega}(n)$ be the subclass of  its finite members.
\end{definition}

\begin{theorem}\label{thm:2}Let $\mathcal{C}$ be a hereditary class of finite or countable posets having width at most two and $\mathcal C_{<\omega}$ be the subclass of its finite members. The following propositions are equivalent.
 \begin{enumerate}[(i)]
  \item $\mathcal{C}$ is bqo;
  \item There exists $n$ such $\mathcal{C}\subseteq \forb(n)$;
  \item There exists $r$ such that for every $P\in \mathcal C_{<\omega}$ whose incomparability graph $\ainc(P)$ is connected, the diameter of the set of vertices of degree at least three in $\ainc(P)$ is at most~$r$.
 \item There exists $k$ such that for every $P\in \mathcal C$ whose incomparability graph  is connected, $\ainc(P)$ consists  of a graph $K$ with diameter at most $k$ with two paths $C_a$ and $C_b$ attached to two distinct vertices $a$ and $b$ of $K$ so that in $P$ every vertex  of $C_{a} \setminus \{a\}$ is below every vertex of $K\setminus \{a\}$, and every vertex  of $C_{b} \setminus \{b\}$ is above  every vertex  of $K\setminus \{b\}$.
  \item $\mathcal C_{<\omega}$ is wqo.
  \end{enumerate}
\end{theorem}

As with Theorem \ref{thm:1}, if we replace  $\mathcal{C}$ by the associated class of bichains or by the associated class of incomparability graphs, one gets a similar characterisation.  A consequence of Theorem \ref{thm:1} is Theorem \ref{thm:korpelainen-lozin}  of Lozin and Mayhill (cf. Theorem 7 in \cite{lozin-mayhill}) (see Comments \ref{comment:lozin} in Section \ref{section:wqo-thm2-thm6}) and Theorem \ref{thm:albert-brignall-Ruskuc-vatter} of Albert et al. (cf. Theorem 9.3. in \cite{albert-brignall-Ruskuc-vatter}).


\begin{figure}
\begin{center}
\includegraphics[width=200pt]{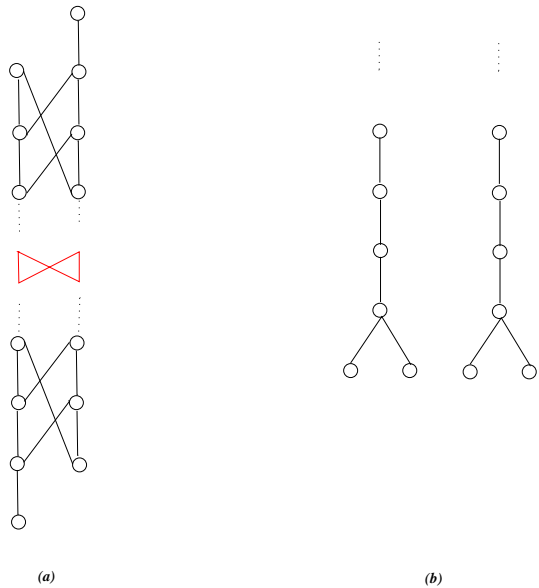}
\end{center}
\caption{The age of $\underline{\mathrm {DF}}_{\infty}$ represented in $(a)$ is wqo and has   infinitely many bounds. The graph represented in $(b)$ is the incomparability graph of $\underline{\mathrm {DF}}_{\infty}$}
\label{fig:wqo-bounds}
\end{figure}

Next, we state two consequence of Theorem \ref{thm:2} related to bounds of hereditary classes of bipartite permutation graphs.

\begin{corollary}Let $\mathcal C$  be a wqo hereditary class of finite  bipartite permutation graphs. Let $G$ be the graph depicted in $(b)$ of Figure $\ref{fig:wqo-bounds}$. Then $\mathcal C$ has finitely many bounds which are permutation graphs if and only if $\age(G)\nsubseteq \mathcal C$.
\end{corollary}
\begin{proof}The forward implication follows from our assumption that $\mathcal C$  is wqo and the fact that $\age(G)$ has infinitely many bounds which are permutation graphs, all double-ended forks. For the converse, assume $\age(G)\nsubseteq \mathcal C$ and let $B$ be the set of bounds of  $\mathcal C$ that are permutation graphs. Then an element of $B$ does not have induced cycles of length at least $5$, and in particular no odd cycles of length at least $5$. Hence, except for the 3-element cycle, all elements of $B$ are bipartite permutation graphs.  Let $B':=B\setminus \{C_3\}$ and set $\mathcal{C}':= \mathcal{C} \cup B'$.  This is a hereditary class of bipartite permutation graphs.  According to Theorem \ref{thm:2},  $\mathcal{C}'$ is wqo if and only if there is an upper-bound on the size of double-ended forks $ \mathrm{DF}_{n}$ it may contain. Thus, if  $\mathcal{C}'$ is not wqo it will contain $\age(G)$ (indeed, if $I$ is an infinite subset of $\NN$, $\age(G)\subseteq \downarrow \{\mathrm{DF}_{n}:n\in I\}$),  hence $\mathcal{C}$ will contain $\age (G)$.
\end{proof}

An other consequence involves the poset $\underline{\mathrm{DF}}_{\infty}$ represented in $(a)$ of Figure \ref{fig:wqo-bounds}.

\begin{corollary} \label{cor-wqo-finite} Let $P$ be a poset of width two. Then $\age(P)$ is wqo with only finitely many bounds if and only if $\age (\underline{\mathrm{DF}}_{\infty}) \not \subseteq \age (P)$.
\end{corollary}

\begin{proof}
The direct implication is due to the fact that $\age (\underline{\mathrm{DF}}_{\infty})$ has  infinitely many bounds (notably the transitive orientations of the complements of the double-ended forks $\mathrm {DF}_n$).  Because, in general,   if $\mathcal C$  is a wqo hereditary class of finite structures with finitely many bounds (like $\age (P)$)  then every hereditary subclass $\mathcal D$ is wqo and has finitely many bounds (indeed,  every bound of $\mathcal D$ not in $\mathcal C$ is necessarily a bound of $\mathcal C$ and since $\mathcal C$ is wqo,  $\mathcal D$ has  only finitely many bounds in $\mathcal C$).
%
For the converse, let $\mathcal B$ be the set of bounds of $\age (P)$ in the class of posets distinct from the  three element antichain. Hence, each   element  of $\mathcal B$ has width at most two. Let $\mathcal C:= \age(P) \cup \mathcal B$.    This is a hereditary class of posets of has width at most two.  According to Theorem \ref {thm:2},  $\mathcal C$ is wqo if and only if there is an upper-bound on the size of the orientations $\underline{\mathrm{DF}}_{n}$ of the complement of  double-ended forks $ \mathrm{DF}_{n}$ it may contain.
Thus, if  $\mathcal C$ is not wqo it will contain $\age (\underline{\mathrm{DF}}_{\infty})$ (indeed, if $I$ is an infinite subset of $\NN$, $\age(\underline{\mathrm{DF}}_{\infty})\subseteq \downarrow \{ \underline{\mathrm{DF}}_{n}:n\in I\}$),  hence $\age (P)$ will contain $\age (\underline{\mathrm{DF}}_{\infty})$.
\end{proof}

An other consequence of   Theorem \ref{thm:2} is this:
\begin{corollary}\label{cor:countably many}There are only countably many wqo hereditary classes made of finite posets having width at most two, resp. made of the associated classes of bichains or the associated classes of bipartite permutation graphs.
\end{corollary}
\begin{proof}
For each $n\in \NN$, $\forb_{<\omega}(n)$ is wqo. Hence, $\forb_{<\omega}(n)$ contains only countably many hereditary subclasses. It follows then that the set of hereditary subclasses contained in some $\forb_{<\omega}(n)$ is countable. These classes are the hereditary subclasses of posets of width at most two. Indeed, if $\mathcal{C}$ is wqo and made of finite posets having width at most two, then by $(v)\Rightarrow (ii)$ of Theorem \ref{thm:2},  $\mathcal{C}\subseteq \forb(n)$. Since the members of $\mathcal{C}$ are finite we infer that $\mathcal{C}\subseteq \forb_{<\omega}(n)$.
\end{proof}

However, there are uncountably many wqo ages of permutation graphs \cite{oudrar-pouzet-zaguia}.

With the help of Theorem \ref{thm:2} and Theorem \ref {thm:infinitepath-kite} below, we give in Theorem \ref{thm:3} a new characterisation of the non wqo ages of  finite bipartite permutation graphs.

\begin{theorem}\label{thm:3}Let $\mathcal{C}$ be an age of finite bipartite permutation graphs. Then $\mathcal{C}$ is not wqo if and only if it contains the age of a direct sum $\bigoplus_{i\in I} \mathrm {DF}_i$ of double-ended forks of arbitrarily large length for some infinite subset $I$ of $\NN$.
\end{theorem}

Note that if $\mathcal{C}$ is a  hereditary class but is not an age, $\mathcal C$ may contain infinite antichains and yet does not contain the   age of a direct sum of double-ended forks of arbitrary large length. A simple example is the class made of graphs which embed into some double-ended fork.
The proof of Theorem \ref{thm:3} follows  from Theorem \ref{thm:infinitepath-kite} below (This is Theorem 6 from \cite{pouzet-zaguia21}). Theorem \ref{thm:infinitepath-kite} is based on the notions of caterpillar and kite.

A graph  $G:=(V,E)$ is a \emph{caterpillar} if the graph obtained by removing from $V$ the vertices of degree one is a path (finite or not, reduced to one vertex or empty). A \emph{comb} is a caterpillar such that every vertex is adjacent to at most one vertex of degree one. Incidentally, a path on three vertices is not a comb.

We now give the definition of a \emph{kite} and we will distinguish three types (see Figure \ref{fig:comb-kite}). A \emph{kite} is a graph extending the   infinite path $\mathrm P_{\infty}$ made of vertices $x_i$ labelled by the nonnegative integers. It has type $(1)$ if it is obtained from $\mathrm P_{\infty}$ by adding a new set of vertices $Y$ (finite or infinite) so that every vertex of $Y$ is adjacent to exactly two vertices of $\mathrm P_{\infty}$ and these two vertices must be consecutive in $\mathrm P_{\infty}$. Furthermore, two distinct vertices of $Y$ share at most a common neighbour in $\mathrm P_{\infty}$.

A \emph{kite of type $(2)$} is obtained from $\mathrm P_{\infty}$ by adding a new set of vertices $Y$ (finite or infinite) so that every vertex of $Y$ is adjacent to exactly three vertices of $\mathrm P_{\infty}$ and these three vertices must be consecutive in $\mathrm P_{\infty}$. Furthermore, for all $x,x'\in Y$, if $x$ is adjacent to $x_i,x_{i+1},x_{i+2}$ and $x'$ is adjacent to $x_{i'},x_{i'+1},x_{i'+2}$  then $i+2\leq i'$ or $i'+2\leq i$.

A \emph{kite of type $(3)$} is  obtained from  $\mathrm P_{\infty}$ by adding a new set of vertices $Y$ (finite or infinite) so that every vertex of $Y$ is adjacent to exactly two vertices of $\mathrm P_{\infty}$ and these two vertices must be at distance two in $\mathrm P_{\infty}$. Furthermore, for all $x,x'\in X$, if $x$ is adjacent to $x_i$ and $x_{i+2}$ and $x'$ is adjacent to $x_{i'}$ and $x_{i'+2}$ then $i+2\leq i'$ or $i'+2\leq i$.


\begin{figure}
\begin{center}
\includegraphics[width=300pt]{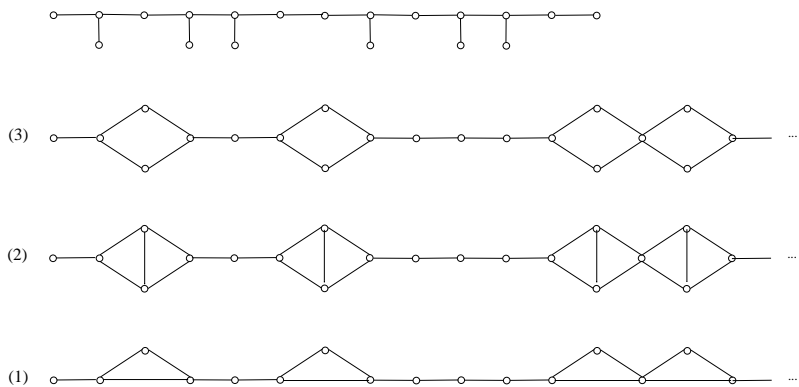}
\end{center}
\caption{A comb and kites.}
\label{fig:comb-kite}
\end{figure}

\begin{theorem}\label{thm:infinitepath-kite}If $G$ is a connected incomparability graph with unbounded diameter, then
\begin{enumerate}[$(1)$]
  \item $G$ embeds an infinite induced path with unbounded diameter starting at any vertex.
  \item If the set of vertices of degree at least three in $G$ has unbounded diameter, then $G$ embeds an induced comb or an induced kite, having  unbounded diameter and infinitely many vertices of degree $3$.
\end{enumerate}
\end{theorem}

With Theorem \ref{thm:infinitepath-kite}, the proof of Theorem \ref{thm:3} goes as follows.
\begin{proof}(of Theorem \ref{thm:3})
Since $\mathcal C$ is not wqo,  it embeds double-ended forks of unbounded length (Theorem \ref{thm:2}). Let $G$ be a graph with $\age(G)= \mathcal C$. We consider two  cases:\\
 $1)$ Some connected component of $G$, say $G_i$, embeds double forks of unbounded length. In this case, the detour of $G_i$, that is the supremum of the lengths of induced paths in $G_i$,  is unbounded. Since $G_i$ is the incomparability graph of a poset of width at most two, its diameter is unbounded  (See Corollary \ref{cor:detour}). In fact, since the vertices of degree $3$ in the forks are end vertices of induced paths, the diameter of the set of vertices of degree $3$ in $G_i$ is unbounded.  Thus from $(2)$ of Theorem \ref{thm:infinitepath-kite},  $G_i$ embeds an induced caterpillar or an induced kite  with infinitely many vertices of degree at least $3$.  Since $G$ is bipartite, it can only embeds a kite of type $(3)$. As it is easy to see, this caterpillar or that kite embeds a direct sum $\bigoplus_{i\in I} \mathrm {DF}_i$ of double-ended forks of arbitrarily large length, as required. \\
 $2)$ If the first case does not hold, there are infinitely many connected components $G_i$, each embedding some double-ended fork $\mathrm {DF}_i$, and the length of these double-ended forks is unbounded. This completes the proof of Theorem \ref{thm:3}.
\end{proof}

The existence  of an infinite antichain in a hereditary class of posets of width two  is equivalent to many other statements. See Proposition \ref{prop:well-founded}  and Corollary \ref{cor:subclasses} in Section \ref{section:testclass}.
 For other classes these equivalences remain conjectures.

Finally, we look at a notion of minimality related to the notion of primality.
We recall that a binary relational structure is {\it prime} (or {\it indecomposable}) if it has no nontrivial modules (see the definition in Section \ref{section:minimality}). A hereditary class $\mathcal{C}$ of finite binary structures  is \emph{minimal prime} if it contains infinitely many prime structures and every proper hereditary subclass has only finitely many prime structures. This  notion appears in the thesis of D. Oudrar \cite{oudrar}, where it is shown that  a  hereditary class which is minimal prime is an age (see Theorem \ref{thm:main1}) and several examples are studied (see Chapter 6 of \cite{oudrar}).

We describe the minimal prime ages associated to posets of width at most two (see the proofs in Section \ref{section:minimality}).

Let $\mathrm P_{\infty}$ be  the one way infinite path. Let  $H$ be the  \emph{half-complete bipartite graph}, that is the graph made of two disjoint independent sets $A:=\{a_i : i\in \NN\}$ and $B:=\{b_i : i\in \NN\}$, where a   pair $\{a_i,b_j\}$ is an edge in $H$ if $i\leq j$. The graphs $H$ and $\mathrm P_{\infty}$ are comparability graphs of the posets represented in Figure \ref {fig:minimalage} as well as incomparability graphs of the posets $P_1$ and $P_2$  represented in Figure \ref{Fig3}.

 \begin{theorem}\label{thm:minimalprimegraph} Let $\mathcal C$ be a hereditary class made of finite posets of width at most two, respectively, of incomparability graphs of posets of width at most two. Then  $\mathcal C$  is  minimal prime if and only if  $\mathcal C$  is either the age
of the posets $P_1$ or $P_2$, or the age of their incomparability graphs, namely $H$ and $\mathrm P_{\infty}$.
\end{theorem}

The second part of the result was obtained by Oudrar \cite{oudrar} Lemma 6.1 page 107.

\begin{figure}[ht]
\begin{center}
\leavevmode \epsfxsize=4in \epsfbox{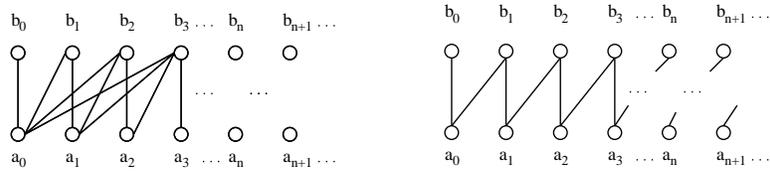}

\end{center}
\caption{The graphs $H$ and $\mathrm P_{\infty}$}
\label{fig:minimalage}

\end{figure}

%
%

\begin{figure}[h]
\begin{center}
\leavevmode \epsfxsize=2.5in \epsfbox{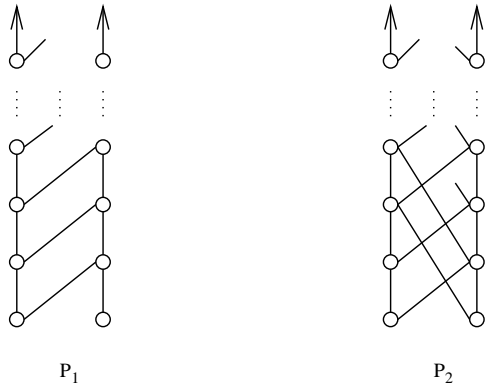}
\end{center}
\caption{Two posets $P_1$ and $P_2$ of width $2$.}
\label{Fig3}
\end{figure}

If $B:= (V, (\leq_1, \leq_2))$ is a bichain, the   \emph{transpose} is $B^{t}:=(V, (\leq_2, \leq_1))$; the \emph{dual} is $B^{*}:= (V, (\leq_1^*, \leq_2^*))$. We set  $o(B):= (V, \leq_1\cap \leq_2)$. Let $B_1:= (\NN, (\leq, \leq_{\omega+\omega}))$ and $B_2:= (\NN, (\leq, \leq_{\omega}))$   be the bichains defined as follows: the order $\leq$ is the natural order on the set $\NN$ of nonnegative integers  while $\leq_{\omega+\omega}$ consists in putting  the odd integers in their natural order before all even integers in their natural order, $\leq_{\omega}$  is defined by $i \leq_{\omega}j$ if $i \leq_{\sigma}j$ where $\sigma $ is the permutation of $\NN$ defined by $\sigma(0):=1$, $\sigma (1):=3$ and for $n\geq 1$, $\sigma(2n):= 2n-2$ and $\sigma(2n+1):= 2n+3$. In fact, we have   $2<_{\omega}0<_{\omega}4<_{\omega} 1 <_{\omega} \cdots  <_{\omega} 2k<_{\omega} 2k-3< \cdots$. Observe that $o(B_1)= P_1$ and $o(B_2):=P_2$. These bichains and the corresponding posets are represented Figures \ref{representationomega-omega} and \ref{representation-plane}.


 \begin{figure}[h]
\begin{center}
\leavevmode \epsfxsize=3.5in \epsfbox{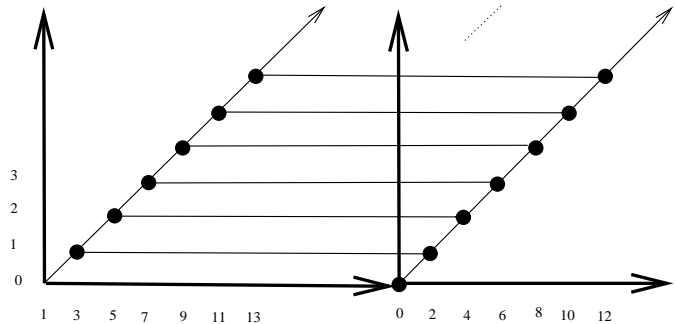}
\end{center}
\caption{A representation  of $B_1$ and  $P_1$}
\label{representationomega-omega}
\end{figure}

 \begin{figure}[h]
\begin{center}
\leavevmode \epsfxsize=3in \epsfbox{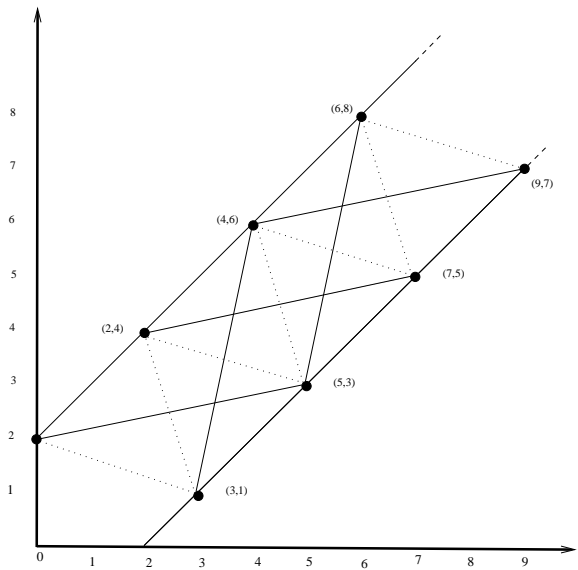}
\end{center}
\caption{A representation of  $P_2$ and $B_2$}
\label{representation-plane}
\end{figure}

\begin{theorem}\label{thm:5} Let $\mathcal C$ be a minimal prime class  of finite bichains $B$ such that the poset $o(B)$ has width at most two.  Then $\mathcal C$  is either the age of one of the  bichains $B_1$, $B_1^t$ or $B_2$.
\end{theorem}

In \cite{pouzet-zaguia2009},  we called a binary relational structure  $R$  \emph{minimal prime} if $R$ is prime and every prime induced structure  with the same cardinality as $R$ embeds a copy of $R$. We described the countable minimal prime graphs with no infinite clique (Theorem 2, p.358). From  our description follows that  a  countable graph $G$ which is the incomparability graph of a poset of width two  is minimal prime if $G$ is either isomorphic to the one way infinite path $\mathrm P_{\infty}$ or to the half-complete bipartite graph $H$. Thus,  $\age(G)$ is minimal prime. The question to know wether or not the age of a relational structure is minimal prime whenever $R$ is minimal prime was considered by Oudrar (\cite{oudrar} Probl\`eme 1, page 99). Her and our  results suggest  that the answer is positive.

Countable prime bichains, and notably those which are perpendicular, have been studied by several authors in several papers \cite{sauer-zaguia, laflamme-pouzet-sauer-zaguia, delhomme-zaguia,delhomme2}. A description of minimal prime bichains, posets and graphs is not known. A characterisation of minimal prime ages of graphs, posets and bichains can be found in \cite{oudrar-pouzet-zaguia}.

\section{Prerequisites}
\subsection{Graphs and  posets}
This paper is about  graphs and posets. Sometimes, we will need  to consider binary relational structures, that is ordered pair $R:=(V,(\rho_{i})_{i\in I})$ where each $\rho_i$ is a binary relation on $V$. The framework of this study is the theory of relations as developed by Fra\"{\i}ss\'e.  At the core  is the notion of embeddability, a quasi-order between relational structures. Several  important notions in the study of these structures, like hereditary classes, ages, bounds, derive from this quasi-order and will be defined when needed. For a wealth of information on these notions see \cite{fraissetr}.

\subsubsection{Graphs}Unless otherwise stated, the graphs we consider are undirected, simple and have no loops. That is, a {\it graph} is a
pair $G:=(V, E)$, where $E$ is a subset of $[V]^2$, the set of $2$-element subsets of $V$. Elements of $V$ are the {\it vertices} of
$G$ and elements of $ E$ its {\it edges}. The {\it complement} of $G$ is the graph $G^c$ whose vertex set is $V$ and edge set
${\overline { E}}:=[V]^2\setminus  E$. If $A$ is a subset of $V$, the pair $G_{\restriction A}:=(A,  E\cap [A]^2)$ is the \emph{graph
induced by $G$ on $A$}.
A \emph{path} is a graph $\mathrm P$ such that there exists a one-to-one map $f$ from the set $V(\mathrm P)$ of its vertices into an
interval $I$ of the chain $\ZZ$ of integers in such a way that $\{u,v\}$ belongs to $E(\mathrm P)$, the set of edges of $\mathrm P$,  if and only if $|f(u)-f(v)|=1$ for every $u,v\in V(\mathrm P)$.  If $I=\{1,\dots,n\}$, then we denote that path by $\mathrm P_n$; its \emph{length} is $n-1$ (so, if $n=2$, $\mathrm P_2$ is made of a single edge, whereas if $n=1$, $\mathrm P_1$ is a single vertex. A graph is $\mathrm P_n$-\emph{free} if it has no $\mathrm P_n$ as an induced subgraph. A \emph{cycle}
is obtained from a finite path $\mathrm P_n:= \{v_{1},...,v_{n}\}$ by adding the edge $\{v_{n},v_{1}\}$ and $n\geq 2$. The integer $n$ is called the
\emph{length} of the cycle. If $G:= (V, E)$ is  a graph, and $x,y$ are two vertices of $G$, we denote by $d_G(x,y)$ the length of the shortest path joining $x$ and $y$ if any, and $d_G:= \infty$ otherwise. This defines a distance on $V$, the \emph{graphic distance}.
The \emph{diameter} of $G$, denoted by  $\delta_{G}$, is the supremum of the set of $d_G(x,y)$ for $x,y\in V$. If $A$ is a subset of $V$, the graph $G'$ induced by $G$ on $A$  is an \emph{isometric subgraph} of $G$  if $d_{G'}(x,y)=d_G(x,y)$ for all $x,y\in A$. The supremum of the length of induced finite paths of $G$, denoted by $D_G$,  is sometimes called the \emph{detour} of $G$ \cite{buckley-harary}. We denote by $D_G(x,y)$ the supremum of the length of the induced path joining $x$ to $y$. Evidently, $d_G(x,y)\leq D_G(x,y)$. In Proposition \ref{prop:oscillation} of Subsection \ref{posetswidth2} we give an upper bound for $D_G$, when $G$ is a permutation graph.

\subsubsection{Posets}Throughout, $P :=(V, \leq)$ denotes an ordered set (poset), that is
a set $V$ equipped with a binary relation $\leq$ on $V$ which is
reflexive, antisymmetric and transitive. We say that two elements $x,y\in V$ are \emph{comparable} if $x\leq y$ or $y\leq x$, otherwise,  we say they are \emph{incomparable}. The \emph{dual} of $P$ denoted $P^{*}$ is the order defined on $V$ as follows: if $x,y\in
V$, then $x\leq y$ in $P^{*}$ if and only if $y\leq x$ in $P$. Let $P :=(V, \leq)$ be a poset. A set of pairwise comparable elements is called a \emph{chain}. On the other hand, a set of pairwise incomparable elements is called an \emph{antichain}. The \emph{width} of a poset is the maximum cardinality  of its antichains (if the maximum does not exist, the width is set to be infinite). Dilworth's celebrated theorem on finite posets \cite{dilworth} states that the maximum cardinality of an antichain in a finite poset equals the minimum number of chains needed to cover the poset. This result remains true even if the poset is infinite but has finite width. If  the poset $P$ has width $2$ and the incomparability graph of $P$ is connected, the partition of $P$ into two chains is unique (picking  any vertex $x$,  observe that the set of  vertices at odd distance from $x$ and the set of vertices at even distance from $x$ form a partition into two chains). This paper is essentially about those posets.   According to  Szpilrajn \cite{szp},  every order on  a set  has a linear extension. Let $P:=(V,\leq)$ be a poset.  A \emph{realizer} of $P$ is a family $\mathcal{L}$ of linear extensions of the order of $P$ whose intersection is the order of $P$. Observe that the set of all linear extensions of $P$ is a realizer of $P$. The \emph{dimension} of $P$, denoted $dim(P)$, is the least cardinal $d$ for which there exists a realizer of cardinality $d$ \cite{dushnik-miller}. It follows from the  Compactness Theorem of First Order Logic that an order is intersection of  at most $n$ linear orders ($n\in \NN$) if and only if every finite restriction of the order has this property. Hence the class of posets with dimension at most $n$ is determined by a set of finite obstructions, each   obstruction is a poset $Q$  of dimension $n+1$  such that  the deletion of any element of $Q$ leaves a poset of dimension $n$; such a poset is said \emph{critical}.  For $n\geq 2$ there are infinitely many critical posets of dimension $n+1$. For $n=2$ they have been  described by Kelly \cite{kelly77}; beyond, the task is considered as hopeless.

We present in Figure \ref{fig:critique3} the critical bipartite posets of dimension $3$ extracted from  the list of Kelly. The comparability graph of the first one is commonly called a  \emph{spider}.
\begin{figure}[h]
\begin{center}
\leavevmode \epsfxsize=3in \epsfbox{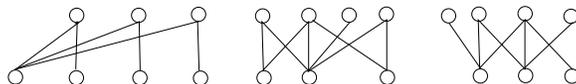}
\end{center}
\caption{Critical bipartite posets  of dimension three.} \label{fig:critique3}
\end{figure}
\subsubsection{Comparability and incomparability graphs, permutation graph}The \emph{comparability graph}, respectively the \emph{incomparability graph}, of a poset $P:=(V,\leq)$ is the undirected graph, denoted by $Comp(P)$, respectively $Inc(P)$, with vertex set $V$ and edges the pairs $\{u,v\}$ of comparable distinct vertices (that is, either $u< v$ or $v<u$) respectively incomparable vertices. A graph $G:= (V, E)$ is a \emph{comparability graph} if the edge set is the set of comparabilities of some order on $V$. From the Compactness Theorem of First Order Logic, it follows that a graph is a comparability graph if and only if every finite induced subgraph is a comparability graph. Hence, the class of comparability graphs is determined by a set of finite obstructions. The complete list of minimal obstructions was determined by Gallai \cite{gallai}. A graph $G:= (V, E)$ is a \emph{permutation graph} if there is a total order $\leq $ on $V$ and a permutation $\sigma$ of $V$ such that the edges of $G$ are the pairs  $\{x, y\}\in [V]^2$ which are reversed by $\sigma$. Denoting by $\leq_{\sigma}$ the set of oriented pairs $(x, y)$ such that $\sigma(x) \leq \sigma (y)$, the graph is the comparability graph of the poset whose order is the intersection of $\leq$ and the opposite of $\leq_{\sigma}$.  Hence, a permutation graph is the comparability graph of  an order intersection of two total orders, that is the comparability graph of an order of dimension at most two \cite{dushnik-miller}. The converse does not hold in general. It holds, if the graph is finite. As it is well known, a finite graph $G$ is a permutation graph if and only if $G$ and $\overline G$ are comparability graphs \cite{dushnik-miller}; in particular, a finite graph  is a permutation graph if and only if its complement is a permutation graph.  Via the Compactness Theorem of First Order Logic, an infinite graph is the comparability graph  of  a poset intersection of two total orders if an only if each finite induced graph is a permutation graph (sometimes these graphs are called permutation graphs, while there is no possible permutation involved). For more about permutation graphs, see \cite{klazar}.

\subsubsection{Lexicographical sum}\label{subsection:lex}Let $I$ be a poset such that $|I|\geq 2$ and let $\{P_{i}:=(V_i,\leq_i)\}_{i\in I}$ be a family of pairwise disjoint nonempty posets
that are all disjoint from $I$. The \emph{lexicographical sum} $\displaystyle \sum_{i\in I} P_{i}$ is the poset defined on
$\displaystyle \bigcup_{i\in I} V_{i}$ by $x\leq y$ if and only if
\begin{enumerate}[(a)]
\item There exists $i\in I$ such that $x,y\in V_{i}$ and $x\leq_i y$ in $P_{i}$; or
\item There are distinct elements $i,j\in I$ such that $i<j$ in $I$,   $x\in V_{i}$ and $y\in V_{j}$.
\end{enumerate}

The posets $P_{i}$ are called the \emph{components} of the lexicographical sum and the poset $I$ is the \emph{index set}.
If $I$ is a totally ordered set, then $\displaystyle \sum_{i\in I} P_{i}$ is called a \emph{linear sum}. On the other hand, if $I$ is
an antichain, then $\displaystyle \sum_{i\in I} P_{i}$ is called a \emph{direct sum}  and will be denoted by $\bigoplus_{i\in I}P_i$ instead.

The decomposition of the incomparability graph of a poset into connected components is expressed in the following lemma which belongs
to the folklore of the theory of ordered sets.

\begin{lemma}\label{lem:folklore} If $P:= (V, \leq)$ is a poset, the order on $P$ induces a total order on the set $Connect(P)$
of connected components of $\Inc(P)$ and $P$ is the lexicographical sum of these components indexed by the chain $Connect(P)$. In
particular, if $\preceq$ is a total order extending the order $\leq$ of $P$, each connected component $A$ of $\Inc(P)$ is an interval of
the chain $(V, \preceq)$.
\end{lemma}

\subsubsection{Initial segment, ideal, final segment, filter, convex set} An \emph{initial segment} of a  poset $P:= (V, \leq)$ is any subset $I$ of $V$ such that $x\in V$, $y\in I$ and $x\leq y$ imply $x\in I$. An \emph{ideal} is
any nonempty  initial segment $J$ of $P$ which is up-directed (that is $x, y\in J$ implies  $x,y\leq z$ for some $z\in J$). If $X$ is a
subset of $V$, the set $\downarrow X:=\{y\in V: y\leq x \; \text{for some}\; x\in X\}$ is the least initial segment containing $X$, we
say that it is \emph{generated} by $X$. If $X$ is a one element set, say $X=\{x\}$, we denote by $\downarrow x$, instead of $\downarrow X$,
this initial segment and say that it is \emph{principal}.
We denote by $\mathbf {I}(P)$, resp. $\mathbf {Id}(P)$,  the set of initial segments, resp. ideals,  of $P$, ordered by inclusion.
\emph{Final segments} and \emph{filters} of $P$ are defined as initial segments and ideals of $P^{*}$.

A subset $X$ of $P$ is \emph{order convex} or \emph{convex} if for all $x,y\in X$, $[x,y]:=\{z : x\leq z\leq y\}\subseteq X$. For instance, initial and final segments are convex. Also, any intersection of convex sets is also convex. In particular, the intersection of all convex sets containing $X$, denoted $Conv_P(X)$, is convex. This is the smallest convex set containing $X$. Note that
\[Conv_P(X)=\{z\in P : x\leq z\leq y \mbox{ for some } x,y\in X\}=\downarrow X\cap \uparrow X.\]
\subsubsection{Order and metric convexities}

We compare the notions of order convexity and metric convexity with respect to the distance on the incomparability  graph of a poset. We present several results that we borrow  from \cite{pouzet-zaguia21} to prove Lemmas \ref{lem3:doublefork},  \ref{lem1:degreethree} and  \ref{lem2:degreethree}.

Let $G:=(V,E)$ be a graph. We equip it with the graphic distance $d_G$. A \emph{ball} is any subset $B_G(x, r):= \{y\in V: d_G(x,y)\leq r\}$ where $x\in V, r\in \NN$. A subset of $V$ is \emph{convex} with respect to the distance $d_G$ if this is an intersection of balls. The \emph{least convex subset} of $G$ containing $X$ is
\[Conv_{G}(X):=\displaystyle \bigcap_{X\subseteq B_G(x,r)}B_G(x,r).\]

Let $X\subseteq V$ and $r\in \NN$. Define
\[B_G(X,r):=\{v\in V : d_G(v,x)\leq r \mbox{ for some } x\in X\}.\]

From  Proposition \ref{cor:orderconvex} it will follow that every ball in a incomparability graph of a poset is order convex and that the graph induced on it is an  isometric subgraph.

We state an improvement  of I.2.2 Lemme, p.5 of \cite{pouzet78} (Lemma 31 of \cite{pouzet-zaguia21}).
\begin{lemma}\label{lem:inducedpath} Let $x,y$ be two vertices of a poset $P$ with $x<y$.  If $x_0, \dots, x_n$ is an induced path in the incomparability graph of $P$ from $x$ to $y$ then   $x_i< x_j$ for all $j-i\geq 2$.
\end{lemma}

The following is Lemma 34 of \cite{pouzet-zaguia21}.

\begin{lemma}\label{lem:convex-boule}Let $P:=(V,\leq)$ be a poset and $G$ be its incomparability graph, $X\subseteq V$ and $r\in \NN$. Then
\begin{align}
&B_G(\downarrow X,r)=(\downarrow X)\cup B_G(X,r)=\downarrow B_G(X,r),\label{eq1}\\
&B_G(\uparrow X,r)=(\uparrow X)\cup B_G(X,r)=\uparrow B_G(X,r),\label{eq2}\\
&B_G(\uparrow X\cap \downarrow X,r)= B_G(\uparrow X,r)\cap B_G(\downarrow X,r),\label{eq3}\\
B_G(\Conv&_P(X),r)=\Conv_P(X)\cup B_G(X,r)=\Conv_P(B_G(X,r)).\label{eq4}
\end{align}
\end{lemma}

The following is Theorem 35  of \cite{pouzet-zaguia21}.

\begin{proposition} \label{cor:orderconvex}Let $P:=(V,\leq)$ be a poset, $G$ be its incomparability graph,  $X\subseteq V$ and $r\in \NN$.
\begin{enumerate}[(a)]
\item If $X$ is an  initial segment, respectively a final segment, respectively an order convex subset of $P$ then $B_G(X,r)$ is an initial segment, respectively a final segment, respectively an order convex subset of $P$. In particular, for all $x\in V$ and $r\in \NN$, $B_G(x,r)$ is order convex;
\item \label{lem:convex-connected}  If $X$ is order convex then the graph induced by $G$ on $B_G(X,r)$ is an isometric subgraph of $G$. In particular, if $X$ is  included into a connected component  of $G$ then the graph induced by $G$ on $B_G(X,r)$ is connected.
\end{enumerate}
\end{proposition}

The order convexity of balls is equivalent to the following inequality (see Corollary 36 of~\cite{pouzet-zaguia21}):

\begin{corollary}\label{lem:monotone-distance}Let $P$ be a poset and let $G$ be its incomparability graph. Then \begin{equation} \label{eq:metric-inequalities}
d_{G}(u,v)\leq d_{G}(x,y) \mbox{ for all }  x\leq u\leq v\leq y  \mbox{ in }  P.
\end{equation}
\end{corollary}

The following is Corollary 37  of \cite{pouzet-zaguia21}.

\begin{corollary}\label{lem:intermediateinducedpath} $\delta_G(X)= \delta_G(Conv_{P}(X))=\delta_G(Conv_{G}(X))$ for every subset $X$ of the poset~$P$.
\end{corollary}

\subsection{Posets of width at most two  and bipartite permutation graphs}\label{posetswidth2}
In this subsection we recall the properties we need about posets of width at most two  and permutation graphs.  We start with a characterisation of bipartite permutation graphs, next we give some properties of the graphic distance and the detour in comparability graphs of posets of width at most two. Some  properties listed below can be found in \cite {pouzet-zaguia21}. A new property, Lemma \ref{lem3:doublefork} is proven.
\subsubsection{Posets of width at most two  and their distances}

We note that a poset $P$ of width at most two  has dimension at most $2$, hence its comparability graph is an  incomparability graph. As previously mentioned, a finite graph $G$ is a comparability and incomparability graph if and only if it is a permutation graph. Incomparability graphs of finite posets of width at most two  coincide with bipartite permutation graphs. For arbitrary posets,  the characterisation  is  as follows (see Lemma 42 of \cite{pouzet-zaguia21}).

\begin{lemma}\label{lem:w2}Let $G$ be a graph. The following are equivalent.
\begin{enumerate}[(i)]
  \item $G$ is  bipartite and the comparability graph of a poset of dimension at most two;
  \item $G$ is bipartite and embeds  no even cycles of length at least six and none  of the comparability graphs of the posets depicted in Figure $\ref{fig:critique3}$.
  \item $G$ is the incomparability graph of a poset of width at most two.
  \item $G$ is a bipartite incomparability graph.
\end{enumerate}
\end{lemma}

We should mention the following result (Lemma 43 in \cite {pouzet-zaguia21},  essentially Lemma 14 from \cite{zaguia2008}) which states that a bipartite permutation graph without cycles must embed a caterpillar. A key observation is that if a vertex has at least three neighboring vertices in $\ainc(P)$, then at least one has degree one. Otherwise, $\ainc(P)$ would have a spider (see Figure \ref{fig:critique3}) as an induced subgraph,  which is impossible.

\begin{lemma}\label{nospider}Let $P$ be a poset of width at most two. Then the following properties
are equivalent.
\begin{enumerate}[(i)]
\item The incomparability graph of $P$ has no cycles of length three or four.
\item The incomparability graph of $P$ has no cycle.
\item The connected components of the incomparability graph of $P$ are caterpillars.
\end{enumerate}
\end{lemma}

We are going to evaluate the detour of connected components of the incomparability graph of a poset of width at most two.

Let $P$ be a poset of width at most two. Suppose that $\ainc(P)$ is connected. In this case,  the partition of $P$ into two chains is unique. An \emph{alternating sequence}   in $P$ is any finite monotonic sequence $(x_0, \dots, x_i, \dots x_n)$ of elements of $P$ (i.e., increasing or decreasing) such that no two consecutive elements $x_i$ and $x_{i+1}$ belong to the same chain of the partition. The integer $n$ is the \emph{oscillation} of the sequence; $x$ and $y$ are its \emph{extremities}.

We recall that the oscillation of an alternating sequence with extremities $x$, $y$ is either $0$ or at most $d_{\ainc(P)}$ (see I.2.4. Lemme p.6  of \cite{pouzet78}). This allows to define the following map.
Let $d_P$ be the map from $P\times P$ into $\NN$ such that.

\begin{enumerate}[$(1)$]
\item $d_P(x,x)= 0$ for very $x\in P$;
\item $d_P(x,y)= 1$ if $x$ and $y$ are incomparable;
\item $d_P(x,y)=2$ if $x$ and $y$ are comparable and there is no alternating sequence from $x$ to $y$;
\item $d_P(x,y)=n+2$ if $n\not =0$ and   $n$ is the maximum of the oscillation of alternating sequences with extremities $x$ and $y$.
\end{enumerate}

We recall a result of \cite{pouzet78} II.2.5 Lemme, p. 6. (Lemma 44 of \cite{pouzet-zaguia21})

\begin{lemma}\label{lem:oscillation-distance}The map $d_P$ is a distance on any poset $P$ of width at most two such that the incomparability graph is connected.  Moreover,  for every $x,y\in P$ the following inequalities hold:
\begin{equation}
0\leq d_{\ainc(P)}(x,y)-d_P(x,y) \leq  2\lfloor d_{\ainc(P)}(x,y)/3\rfloor.
\end{equation}
\end{lemma}
The distance $d_P$ can be different from $d_{\ainc(P)}$. For some bichains these two distances coincide, in that case we say that they are \emph{regular}. Two examples, namely $\QQ_{\pi,2}$ and $\RR_{\pi,2}$,  are given below.

We give a slight  improvement of \cite{pouzet78} I.2.3. Corollaire, p. 5.(Lemma 45 of \cite{pouzet-zaguia21}).

\begin{lemma}\label{lem:oscillation2}
Let $P$ be poset of width at most two such that $\ainc(P)$ is connected.  Let $n\in \NN$,  $r\in \{0,1\}$ and   $x, y\in P$ such that $\ainc(P)$ contains an  induced   path of length $3n+r$ and extremities $x$ and $y$. If  $r\not =1$ and $n\geq 1$(resp. $r=1$ and $n\geq 2$) then there is an alternating sequence with extremities $x$,$y$ and oscillation $n$ (resp. $n-1$).
\end{lemma}


From Lemma \ref{lem:oscillation-distance}, the oscillation between two vertices $x$ and $y$ of $P$  is bounded above. With this lemma, the length of induced paths between $x$ and $y$ is bounded  too, that is the detour $D_{\ainc(P)} (x,y)$ is an integer.  In fact we have (see Proposition  46 of \cite{pouzet-zaguia21}):

\begin{proposition}
Let $P$ be poset of width at most two such that $\ainc(P)$ is connected and let $x,y\in P$.
Then:

\begin{enumerate}[$(1)$]\label{prop:oscillation}
\item $d_{\ainc(P)}(x,y)=d_P(x,y)= D_{\ainc(P)}(x,y)$  if either $x=y$, in which case this common value is $0$, or $x$ and $y$ are incomparable, in  which case this common value is $1$.
\item $d_{\ainc(P)}(x,y)\geq d_P(x,y)\geq \lfloor  D_{\ainc(P)}(x,y)/3 \rfloor +\epsilon$ where $\epsilon=1$ if  $D_{\ainc(P)}(x,y)\equiv 1 \mod 3$ and $\epsilon=2$ otherwise.
\end{enumerate}

\end{proposition}


As a consequence (see Corollary 47 of \cite{pouzet-zaguia21})
\begin {corollary}\label{cor:detour} If  the incomparability graph $G$  of a poset of width at most two is connected, then the detour of $G$ is bounded if and only if the diameter is bounded. In fact:
\[\delta_G\leq D_G\leq 3\delta_G -1. \]
\end{corollary}

\subsubsection{On the diameter of the set of vertices of degree at least $3$ in the incomparability graph of a poset of width at most two}\label{subsection:diameter}

\begin{lemma}\label{lem3:doublefork}Let $G$ be a bipartite graph.
\begin{enumerate}[$(1)$]
  \item If $G$ has two vertices of degree at least $3$ and at distance $d\geq 4$, then $G$ embeds an induced double-ended fork of length at least $d-4$.
  \item If $G$ is the incomparability graph  of a poset of width at most two containing an induced double-ended fork of length at least $d$, the distance in $G$ between the extremities of the double-ended fork is at least $\frac{d}{3}+1$.
\end{enumerate}
\end{lemma}
\begin{proof}
\begin{enumerate}[$(1)$]
  \item Let $x,y$ be two vertices of degree at least $3$ such that $d_G(x,y)=d\geq 4$ and let $x=v_0,\dots,v_{d}=y$ be an induced path of length $d$ joining $x$ and $y$. \\
      \textbf{Claim:} Every neighbour of $x$ but $v_1$, respectively every neighbour of $y$ but $v_{d-1}$, is not in $\{v_0,\cdots,v_{d}\}$.\\
      \textbf{Proof of Claim:} Let $x_1,x_2$, respectively $y_1,y_2$, be neighbours of $x$, respectively be neighbours of $y$, not in $\{v_0,\dots,v_{d}\}$. Note that since $G$ has no triangles $x_1$ and $x_2$ are not adjacent and so are $y_1$ and $y_2$. Moreover,  $x_i$, for $i=1,2$, is not adjacent to $v_1$. Similarly $y_i$, for $i=1,2$, is not adjacent to $v_{d-1}$. Furthermore, $x_i$, for $i=1,2$, is adjacent to at most two vertices in $\{v_0,\cdots,v_{d}\}$ because otherwise the distance between $x$ and $y$ would be less than $d$. Moreover, and since $G$ has no odd cycles, $x_i$, for $i=1,2$, can only be adjacent to $v_0$ and $v_2$. By symmetry we have that $y_i$, for $i=1,2$, can only be adjacent to $v_{d-2}$ and $v_{d}$.\hfill $\Box$

      We now claim that there exists $j\in \{0,2\}$ and $k\in \{d-2,d\}$ such that $v_j$ and $v_k$ are the extremities of a double-ended fork proving (1). Indeed, If some $x_i$ is adjacent to $v_2$ set $j=2$. By symmetry set $k=d-2$ if some $y_i$ is adjacent to $v_{d-2}$. Otherwise set $j=0$ and $k=d$.
  \item This is an immediate consequence of the inequality (2) of Proposition \ref{prop:oscillation} applied to a poset $P$ such that $\ainc(P)=G$. Indeed, let $x,y$ be the extremities of the double-ended fork. We have $d_G(x,y)\geq \lfloor \frac{D_G(x,y)}{3} \rfloor +\epsilon$ where $\epsilon=1$ if $D_P(x,y)\equiv 1 \mod 3$ and $\epsilon=2$ otherwise. Since  $D_G(x,y) \geq d$ and $\epsilon \geq 1$, we have $\frac{D_G(x,y)}{3}+\epsilon \geq \frac{d}{3} +1$.
\end{enumerate}
\end{proof}

\begin{lemma}\label{lem1:degreethree} Let  $P:= (V, \leq)$ be a poset of width at most two whose incomparability graph is  connected.  Let $I$ be the initial segment of $P$ generated by the  set of vertices of degree at least $3$.
If $C$ is a nontrivial connected component of $\ainc(P)_{\restriction V\setminus I}$ then $C$ is a path and all other connected component are below $C$ in $P$, in particular they are trivial. If $C$ has at least two vertices then some  vertex $x$ of $C$ with degree $1$ is  connected to some  vertex belonging to $I$ and this vertex is unique. If $C$ has at least four vertices, then one and only one vertex of $C$ is connected to some vertex of $I$ and this vertex is unique.
\end{lemma}
\begin{proof}Let $G:= \ainc(P)$. If $I$ is empty, that is every vertex has degree at most $2$, then,  being connected, $G$ is a path. If $I$ is nonempty,  the graph $G_{\restriction V\setminus I}$ is not necessarily connected, but its connected components are paths and according to Lemma \ref{lem:folklore} they are  totally ordered. Suppose that there are two distinct components  $C$ and $C'$. We may suppose  $C< C'$ (that is,  every $x\in C$ is below, in $P$,  every $x\in C'$). Let $x\in C, x'\in C'$ such that the distance  $d_G(x,x')$ is minimum. Let $x_0, \dots, x_n$  be a path in $G$ joining $x$ to $x'$ such that $n=d_G(x,x')$.\\
\textbf{Claim:} $n=2$ and $x_1\in I$.\\
\textbf{Proof of Claim.} Clearly, $n\geq 2$. Necessarily $x_1\in I$ otherwise $x_1\in C$ contradicting the minimality of the distance $d_G(x,x')$. Similarly, $x_{n-1} \in I$. If $n\geq 3$, then from Lemma \ref{lem:inducedpath}, $x_0\leq x_{n-1}$ in $P$. Since $I$ is an initial segment and  $x_{n-1} \in I$ we have   $x_0\in I$,  which  is impossible. Hence $n<3$ and thus $n=2$. This completes the proof of the claim.\hfill       $\Box$

This ensures that $C$ is trivial. Indeed, if $C$ contains at least two vertices, then it contains some vertex $y$ connected to $x$ by an edge of $G$, that is $y$ is incomparable to $x$.  Since $x$ and $x_1$  are incomparable and $P$ has width at most $2$, $x_1$ and $y$ are comparable. Since $x_1\in I$  and $I$ is an initial segment,  $y\leq x_1$  is impossible; thus $x_1< y$. But, since $C\leq C'$, $y<x'$ and thus $x_1<x'$, which is impossible since by definition $x_1$ is incomparable to $x'$.

Consequently, if $C$ is nontrivial, all other components, if any, are trivial and below $C$. Suppose that $C$ is nontrivial. Since $I$ is nonempty, $C$ cannot be a two ways infinite path, hence either $C$ is a one way infinite path or a finite one and we may label the vertices as $x_0, \dots, x_i, \dots x_n$ with $x_i< x_{i+2}$ for all $i$.  If $x$ is any element of $C$ which is not of degree $1$, it cannot be incomparable to an element of $I$, otherwise it would have degree at least $3$ in $G$. Thus $x$ is above all elements of $I$. Since $G$ is connected,  only an element of degree $1$ of $C$  can be connected  to some element $y$ of $I$ and this element $y$ is unique (otherwise the vertex of degree $1$ in $C$  would have degree at least $3$ in $G$). If $C$ has three vertices, say $x_0, x_1, x_2$, and if $x_2$ is connected to some element $y'$ of $I$ then $y=y'$. Indeed, we have  $x_0<x_2$. If $y'\not =y$,  $y'$ cannot be connected to $x_0$, otherwise, this vertex would have degree at least $3$ in $G$ which is impossible. Thus $y'$ and $x_0$ are comparable and necessarily $y'<x_0$.  Since $ x_0<x_2$,    it follows $y'<x_2$ which is impossible.  If $C$ has at least four vertices, then $x_0$ is connected to a unique vertex of $I$ and all other elements of $C$ are above all elements of $I$ (indeed, let $n\geq 1$; if $x_n$ is not an end vertex of $C$ it is above all elements of $I$. If $x_n$ is an end vertex, then  $n\geq 3$.  Hence $x_n$ is above $x_0$ and $x_1$ thus is above all elements of $I$).  This completes the proof of the lemma.
\end{proof}

\begin{lemma}\label{lem2:degreethree}
Let  $P:= (V, \leq)$ be a poset of width at most two whose incomparability graph is  connected. If the diameter of $d^{\geq 3} (\ainc (P))$, the set of vertices of degree at least $3$,  is bounded by some integer $d$ then there is a connected induced subgraph  $K$  of $\ainc(P)$  with diameter at most $d+4$ such that $\ainc(P)$ consists of $K$ and  at most two paths attached to two distinct vertices of $K$.
\end{lemma}
\begin{proof}Let $G:= \ainc(P)$ and set $A:= d^{\geq 3} (G)$. If $A$ is empty, then $G$ is a path, finite or infinite, whereas if  $A$ is a singleton, then $G$ is a path with a pending vertex; in both cases, the conclusion of the lemma holds.  So we may suppose that $A$ has at least two elements. Let $r\geq 1$ be an integer.\\
\textbf{Claim 1:} $\delta_{G} (Conv_P(A))\leq d$, $B_G(Conv_P(A),r)$ is a connected subgraph of $G$ and  $$\delta_G(B_G(Conv_P(A),r))\leq \delta_{G} (Conv_P(A))+2r. $$
\textbf{Proof of Claim 1.}According to Corollary  \ref{lem:intermediateinducedpath} $\delta_{G} (Conv_P(A))= \delta_G(A)$. The first inequality follows. The second part of the claim follows from  (\ref{lem:convex-connected}) of Proposition \ref{cor:orderconvex}. The third part follows from the triangular inequality.  \hfill $\Box$

Let $K$ be the graph induced by  $G$  on $B_G(Conv_P(A),2)$.

\noindent \textbf{Claim 2:} The graph induced on $V\setminus B_G(Conv_P(A),2)$ is made of at most two paths, say $L^{-}$ and $L^{+}$, each one being linked by an  edge  to one and only one vertex of $B_G(Conv_P(A),2)$, these two vertices being distinct.\\
\textbf{Proof of Claim 2.} We notice at once $\delta_G(B_G(Conv_P(A),2))\leq \delta_{G} (Conv_P(A))+4\leq d+4$. Let $I_{2}$, resp. $F_{2}$  be the set of vertices at distance at most $2$ of $\downarrow A$, resp $\uparrow A$. According to $(3)$ of Lemma \ref{eq3} of \ref{lem:convex-boule} applied to $X= A$ and $r=2$ we obtain $I_{2} \cap F_{2}=B_G(Conv_P(A),2)$. Hence,  $V\setminus B_G(Conv_P(A),2)=(V\setminus I_2) \cup  (V\setminus F_2)$. Since $V\setminus I_2\ \subseteq V\setminus I$ we may apply Lemma \ref {lem1:degreethree} and deduce that there is a path in $V\setminus I_2$. Similarly,  we obtain a path in $V\setminus F_2$. These are the paths $L^{+}$ and $L^{-}$.
\end{proof}

\subsection{The countable universal poset of width two}

We recall the existence of a universal poset of width at most two. We describe the incomparability graph of a variant of this poset more appropriate for our purpose.

 Let $\RR_{\pi,2}$ be the set of couples $(x,i)$ with $x\in \RR$ and $i\in\{0,1\}$ ordered as follows: $(x, i)\leq (y,j)$ if either $i=j$  and $x \leq y$ in $\RR$ or $i \neq j$ and $x +\pi \leq y$ in $\RR$. Let $\QQ_{\pi,2}$ be the restriction of $\RR_{\pi,2}$ to $\QQ\times \{0,1\}$. These two posets have width two and their incomparability graphs are connected. In \cite{pouzet78}, the first author proved that $\QQ_{\pi,2}\cdot  \QQ$, the lexicographical sum of copies of $\QQ_{\pi,2}$ indexed by the chain $\QQ$ of rational numbers is universal among countable posets of width two, that is, every finite or  countable poset of width two embeds into $\QQ_{\pi,2} \cdot  \QQ$. From this fact follows that $\QQ_{\pi,2}\cdot \QQ$, $\QQ_{\pi,2}$ and $\RR_{\pi,2}$ have the same age. Hence,  every poset $P$ at most countable, of width at most two, such that $\ainc (P)$ is connected embeds in $\QQ_{\pi,2}$ as well as in $\RR_{\pi,2}$.

Let $C:=(X,\preceq)$ be a totally ordered set. Set $C_{\ZZ}:=(X\times \ZZ,\leq)$ where $\leq$ is  the binary relation on $X \times \ZZ$ defined as follows:
\[(x,n)\leq (x',n') \mbox{ if } x\preceq x' \mbox{ and } n\leq n' \mbox{ or } x\succ x' \mbox{ and } n'\geq n+2.\]

\begin{lemma}\label{lem:bichain-coding}
\begin{enumerate}[$(1)$]
\item The binary relation $\leq$ is an order relation.
\item Two elements $(x,n)$ and $(x',n')$ are incomparable if and only if either ($x<x'$ and $n'=n-1$) or ($x'<x$ and  $n=n'-1$).
\item The poset $C_{\ZZ}$ has width at most two.
\item Every local isomorphism of the total order $C$ induces a local isomorphism of $C_{\ZZ}$.
\end{enumerate}
\end{lemma}
\begin{proof}
\begin{enumerate}
\item We should mention at first that $\leq$ is the union of the coordinate wise order, say $\leq_1$  on the cartesian product $X\times \ZZ$ and a strict order, say $\leq_2$. In particular, $\leq$ is reflexive. We  prove that $\leq$ is antisymmetric. Let $u:=(x,n)$ and $u':= (x',n')$. Suppose $u\leq u'$ and $u'\leq  u$. We may suppose $u\leq_1 u'$ and $u'\leq_2 u$. If $u\not =u'$, this is clearly impossible. We now prove that $\leq$ is transitive. Let $u:=(x,n)$,  $u':= (x',n')$ and $u'':= (x'',n'')$.  Suppose $(x,n)\leq (x',n')\leq (x'',n'')$. If $u'\leq_1 u''$ and $u'\leq_2 u''$ then $x\preceq x'$ and $x''\prec x'$ and $n+2\leq n''$. Since $\preceq$ is a total order, we have $x\preceq x''$ or $x'' \succ x$ and in both cases we deduce that $(x,n)\leq (x'',n'')$. Else $u'\leq_2 u''$ and $u'\leq_1 u''$ Then $x'\prec x$ and $x'\preceq x''$ and $n+2\leq n''$. Since $\preceq$ is a total order, we have $x\preceq x''$ or $x'' \succ x$ and in both cases we deduce that $(x,n)\leq (x'',n'')$.
\item Suppose that  $(x,n)$ and $(x',n')$ are incomparable. Since $x$ and $x'$ are distinct, we may suppose $x<x'$. Necessarily, $n' =n-1$.
\item ${C_\ZZ}_{\restriction (X\times 2\ZZ)}$ and ${C_\ZZ}_{\restriction (X\times (2\ZZ+1))}$ are two chains covering $C_{\ZZ}$.
\item Let $\varphi$ be a local isomorphism of $C$. Define the function $f: X\times \ZZ \rightarrow X\times \ZZ$ as
    $f(x,n)=(\varphi(x),n)$. This is clearly a local isomorphism of $C_{\ZZ}$.
\end{enumerate}
\end{proof}

 Set $C_{[-n,n]}:={C_{\ZZ}}_{\restriction X\times [-n,n]}$. If $C$ is  the real interval $[0,2\pi[$ with the natural order, set $D:=[0,2\pi[_{\ZZ}$ and $D(n):=[0,2\pi[_{[-n,n]}$ for $n\in \NN$. Hence $D= \bigcup_{n\in \NN} D(n)$.

\begin{lemma}The posets $D$ and $\RR_{\pi,2}$ are isomorphic.
\end{lemma}
\begin{proof}Let $f$ be the map from $[0,2\pi[\times \ZZ$ into $ \RR \times\{0,1\}$ defined as follows: $f(r,n):=(n\pi+r,n\mod 2)$ and note that $f$ is well defined. We claim that $f$ is an isomorphism of $C_{\ZZ}$ onto $\RR_{\pi,2}$. We first prove that $f$ is a bijection. Let $n,n'\in \ZZ$ and $r,r'\in[0,2\pi[$ and suppose that $f(r,n)=f(r',n')$, that is, $(n\pi+r,n\mod 2)=(n'\pi+r',n'\mod 2)$. This means that $n'-n$ is even and  $(n-n')\pi=r'-r$. From $r,r'\in[0,2\pi[$ it follows readily that $n=n'$ and $r=r'$ proving that $f$ is one-to-one. Let $(x,i)\in \RR_{\pi,2}$ and let $n$ be the unique integer such that $n\pi \leq x< (n+1)\pi$ in $\RR$ and set $r:=x-n\pi$. Hence $x=n\pi +r$ with $r\in [0,2\pi[$. Let $i\in \{0,1\}$. If $i\equiv n\mod 2$, then $f(r,n)=(x,i)$. Else $f(r+\pi,n-1)=(x,i)$. This proves that $f$ is onto. Next we prove that $f$ and $f^{-1}$ are order preserving. Let $n,n'\in \ZZ$ and $r,r'\in[0,2\pi[$ be such that $(r,n)\leq (r',n')$ in $C_{\ZZ}$. Suppose $r\leq r'$ and $n\leq n'$. Then $n\pi +r\leq n'\pi +r'$. If $n'\not \equiv n \mod 2$, then $n<n'$ and $n'\pi +r'-(n\pi +r)=(n'-n)\pi+r'-r\geq \pi$ and hence $f(r,n)\leq f(r',n')$ in $\RR_{\pi,2}$. We now suppose that $r\succ r'$ and $n'\geq n+2$. Then $n'\pi +r'-(n\pi +r)=(n'-n)\pi+r'-r\geq 2\pi+r'-r\geq \pi$ and hence $f(r,n)\leq f(r',n')$ in $\RR_{\pi,2}$. This proves that $f$ is order preserving.\\
In order to prove that $f$ is an isomorphism, it is enough to prove that $f$ preserves the incomparability relation of $\leq$ on $D$. Let $n,n'\in \ZZ$ and $r,r'\in[0,2\pi[$ such that $(n, r)$ and $(n'r')$ are incomparable. With no loss of generality, we may suppose $r<r'$. According to Item (2) of Lemma \ref{lem:bichain-coding}, $n'=n-1$. We have $f(r,n)= (n\pi+r, n \mod(2))$ and $f(r',n')= (n-1)\pi+r',n-1, \mod(2))$. Since $\vert n\pi+r-(n-1)\pi+r'\vert= \vert \pi+(r-r')\vert <\pi$, $f(r,n)$ and $(r',n')$ are incomparable.
This completes the proof of the lemma.
\end{proof}

\begin{corollary}\label{cor:embed-D} Every poset $P$ at most countable, of width at most two, such that $\ainc (P)$ is connected embeds in $D$.
\end{corollary}

\begin{lemma}If $C$ is countable and dense in $[0,2\pi[$, $C_{\ZZ}$ is equimorphic to $\QQ_{\pi,2}$.
\end{lemma}
\begin{proof}The poset $\QQ_{\pi,2}$ embeds in $\QQ_\ZZ$. Since $C$ is dense in $[0,2\pi[$ the chain $\QQ$ embeds in $C$. Hence, $\QQ_{\pi,2}$ embeds in $C_{\ZZ}$. \\
Since $\QQ_{\pi,2}$ is universal among connected countable posets of width at most $2$ we infer that $C_{\ZZ}$ embeds in $\QQ_{\pi,2}$.
\end{proof}

We describe the balls of $\ainc (\RR_{\pi, 2})$ and their counterpart in $\ainc(D)$  for the graphic distance.

 \begin{lemma}
In $\ainc (\RR_{\pi, 2})$, the ball $B((0,0), 1)$ is equal to $\{(0,0)\} \cup \{(x,1):  -\pi <x<\pi \}$. For $n\geq 2$,
$B((0,0), n)=B((0,0), n-1) \cup \{(x, n \mod 2):  -n \pi <x<n \pi \}$.

In particular, $B((0,0), n) \subseteq D(n)$.
  \end{lemma}

\begin{proposition} \label{bounded-diameter}
\begin{enumerate}[$(1)$]

 \item For each integer $n\geq 1$, the incomparability graph of $D(n)$ has a bounded diameter;

 \item For every integer $n\geq 1$ there is an integer  $m(n)$ such that every poset $P$ at most countable and of width at most two and such that the diameter of $\ainc (P)$ is at most $n$ embeds in $D(m(n))$. In particular every connected bipartite permutation graph $G$ of diameter bounded by $n$ embeds into the incomparability graph of the poset $D(m(n))$.
 \end{enumerate}
\end{proposition}
\begin{proof}
Let $P$ be  a poset at most countable and of width at most two whose incomparability graph is connected. It follows from Corollary \ref{cor:embed-D} that $P$ embeds in $D$. Let $f$ be such an embedding. This is an embedding of $\ainc(P)$ in $\ainc(D)$.   If the diameter of $\ainc (P)$ is at most $n$ then the diameter of the range of $f$ is at most $n$. We may suppose that its image is included into the ball $B((0,0), n)$. Hence it is included into $D(n)$.
\end{proof}

 \begin{figure}[h]
\begin{center}
\leavevmode \epsfxsize=3in \epsfbox{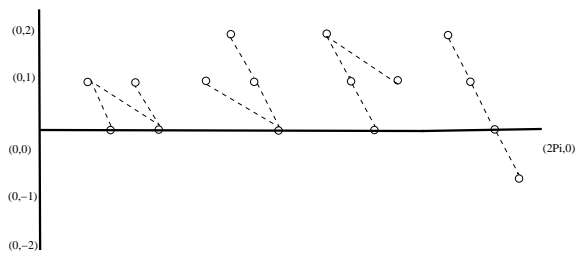}
\end{center}
\caption{Embeddings of $P_4$ into $\Inc(D_2)$.} \label{fig:p4}
\end{figure}

\begin{figure}[h]
\begin{center}
\leavevmode \epsfxsize=3in \epsfbox{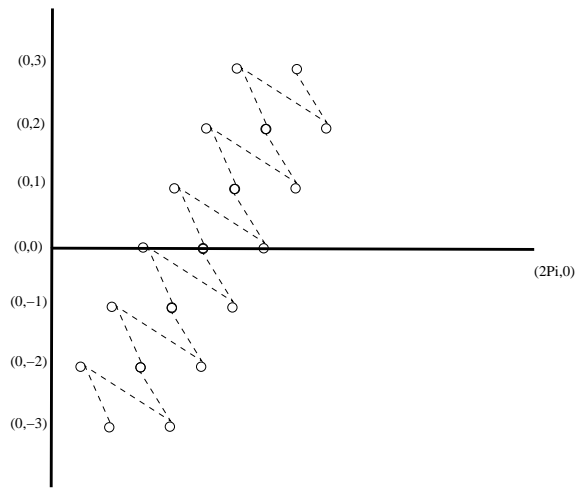}
\end{center}
\caption{An embedding of $P_{18}$ into $\Inc(D_3)$.} \label{fig:path-D3}
\end{figure}

\begin{lemma}Let $n$ be a positive integer. Denote by $f(n)$ the largest positive integer $k$ such that $\mathrm P_k$ embeds into $\ainc(D(n))$. Then $f(n)=6n$ for $n\geq 1$.
\end{lemma}
\begin{proof}Let $n\geq 1$ be an integer. We first exhibit an incomparability path of length $6n$ proving that $f(n)\geq 6n$ (see Figure \ref{fig:path-D3} for an example when $n=3$). Let $(x_k)_{0\leq k\leq 6n}$ be the sequence of vertices of $D_n$ defined as follows. Set $x_0=(\frac{\pi}{n},-n), x_1=(\frac{\pi}{n}-\frac{1}{2n},-n+1),x_2=(\frac{\pi+1}{n},-n),x_3=(\frac{\pi}{n}+\frac{1}{2n},-n+1)$. For $k\geq 1$ set $x_{3k+1}=x_{1}+(\frac{1}{2k},k)$, $x_{3k+2}=x_{2}+(\frac{1}{2k},k)$ and $x_{3k+3}=x_{3}+(\frac{1}{2k},k)$. It is easily seen that the sequence $(x_k)_{0\leq k\leq 6n}$ is an incomparability path in $D_n$ of length $6n$. Thus $f(n)\geq 6n$.\\
Next we prove that $f(n)\leq 6n$. We will prove that if $R:=(r_i,l_i)$  for $i\geq 0$ is an incomparability path in $D_n$, then $R$ has at most three edges in $[0,2\pi[\times \{j,j+1\}$ for all $-n\leq j\leq n-1$. Since there are $2n$ such sets, the maximum number of edges of $R$ is $6n$ and we are done. Suppose for a contradiction that there exists $-n\leq j\leq n-1$ for which $R$ has four distinct edges in $[0,2\pi[\times \{j,j+1\}$. Say $\{(r_k,j),(r'_k,j+1)\}$, for $1\leq k\leq 4$, are the edges of $R$ in $[0,2\pi[\times \{j,j+1\}$. We may assume without loss of generality that $r_1\leq r_2\leq r_3\leq r_4$. Then $r'_k<r_k$ for all $1\leq k\leq 4$. Since $(r_4,j)$ has degree at most two in $R$ we infer that $r'_1=r'_2=r'_3$ or $r'_2=r'_3=r'_4$. We only consider the case $r'_1=r'_2=r'_3$ as the other case can treated similarly. Hence, $(r'_1,j+1)$ is adjacent to $(r_k,j)$ for all $1\leq k\leq 4$. Hence, $r_1=r_2=r_3$ or $r_2=r_3=r_4$. But this is impossible since we assumed that $R$ has four distinct edges in $[0,2\pi[\times \{j,j+1\}$. This completes the proof of the lemma.
\end{proof}

\begin{lemma} \label{dm-multichainable}For every nonnegative integer $n$,  the poset $D (n)$ and its incomparability graph $\Inc(D(n))$ are multichainable.
\end{lemma}
\begin{proof} Apply Item (4) of Lemma \ref{lem:bichain-coding}.
\end{proof}

\section{Well-quasi-ordered and better-quasi-ordered hereditary classes and a proof of Theorems \ref{thm:1} and \ref{thm:2}.}\label{section:wqo-thm2-thm6}
We recall the notion of well-quasi-order. We refer to the founding work of Erd\"os and Rado\cite{erdos-rado} and Higman \cite{higman}. Then,  we recall the notion of better-quasi-order invented by Nash-Williams \cite{nashwilliams1, nashwilliams2}.

A quasi-ordered-set  (quoset) $Q$  is \emph{well-quasi-ordered}(wqo), if every infinite sequence of elements of $Q$ contains an infinite increasing subsequence. If $Q$ is an ordered set, this amounts to say that every nonempty subset of $Q$ contains finitely many minimal elements (this number
being non zero). Equivalently, $Q$ is wqo if and only if it contains no infinite descending chain and no infinite antichain.
Among classes of structures which are wqo under the embeddability quasi-order some remain  wqo when the structures are labelled by the elements of a quasi-order. Precisely, Let $\mathcal{C}$ be a class of relational structures, e.g., posets, and $Q$ be  a quasi ordered set or a poset. If $R\in \mathcal C$, a \emph{labelling of $R$ by $Q$} is any map $f$ from  the domain of $R$ into $Q$. Let   $\mathcal{C}\cdot Q$ denotes the collection of $(R,f)$ where $R\in \mathcal{C}$ and $f: R\rightarrow Q$ is  a labelling. This class is quasi-ordered  by $(R,f)\leq (R',f')$ if there exists an embedding $h: R\rightarrow R'$ such that $f(x)\leq (f'\circ h)(x)$ for all $x\in R$. We say that $\mathcal{C}$ is \emph{very well quasi-ordered} (vwqo for short) if for every finite $Q$, the class $\mathcal{C}\cdot Q$ is wqo. The class $\mathcal{C}$ is \emph{hereditary wqo} if $\mathcal{C}\cdot Q$ is wqo for every wqo $Q$. The class $\mathcal{C}$ is $n$-wqo if for every $n$-element poset $Q$, $\mathcal{C}\cdot Q$ is wqo.  The class $\mathcal{C}$ is $n^{-}$-w.q.o. if the class of $(R,a_1,\cdots,a_n)$ where $R\in \mathcal{C}$ and $a_1,\cdots,a_n\in R$ is wqo.

We do not know if these four notions are different. In the case of posets of width two we prove that they are identical.

We recall the following result (Proposition 2.2 of \cite{pouzet72}).

\begin{theorem}\label{thm:bounds} Provided that the signature $\mu$ is bounded, the cardinality of bounds of  every hereditary and hereditary wqo subclass of  $\Omega_{\mu}$ is bounded.  \end{theorem}

The notion of better-quasi-ordering is a tool for proving that classes of infinite structures are wqo. The operational definition of bqo is not intuitive. Since we are not going to  prove properties of bqos, we use  the following  definition,  based on the idea of labelling considered above. Let $Q$ be a quasi-ordered set and $Q^{<\omega_1}$,  the set of maps $f: \alpha \rightarrow Q$, where $\alpha$ is any countable ordinal. If $f$ and $g$ are two such maps, we set $f\leq g$ if there is a one-to-one preserving map $h$ from the domain $\alpha$ into the domain $\beta $ of $g$ such that $f(\gamma)\leq g(h((\gamma))$ for all $\gamma< \alpha$. This relation is  a quasi-order; the quasi-ordered set  $Q$ is a \emph{better-quasi-order} if $Q^{<\omega_1}$ is wqo. Bqo's are wqo's. Indeed, according to the famous Higman's Theorem on words, \cite{higman}, a set $Q$ is wqo if and only if  the set $Q^{<\omega}$ of finite sequences of elements of $Q$ is wqo. As  wqo's, finite sets and  well-ordered sets are bqo's (do not try to prove it using  the  definition of bqo given here), finite unions, finite products, subsets and images of bqo's are bqo's.  But, contrarily to wqo's,  if $Q$ is bqo then the set $\mathbf{I} (Q)$ of initial segments of $Q$  is bqo. Instead of labelling ordinals, we may label chains and compare them as above. We will need the following result of Laver \cite{laver} (Theorem,  page 90).

\begin{theorem}\label{thm:laver}
The class of countable chains labelled by a bqo is bqo.
\end{theorem}

A relational structure $R$ is \emph{almost multichainable} if its domain $V$ is the disjoint union of a finite set $F$ and a set  $L\times K$ where  $K$ is a  finite set, for which there is a total order $\leq$ on $L$, satisfying  the following condition:

 $\bullet$ For every local isomorphism $h$  of the chain $C:= (L, \leq)$  the map $(h, 1_K)$ extended by the identity on $F$ is a local isomorphism of $R$ (the map $(h, 1_K)$ is defined by $(h, 1_K)(x, y):= (h(x), y)$ ).

We state

\begin{proposition}\label{multichainablewqo}The age of an almost multichainable structure  is hereditary wqo.
\end{proposition}

It was proved in  $1.$  of Theorem 4.19 p.265 of \cite{pouzet2006} that such an age is very well wqo, that is remains wqo if its members are labelled by a finite wqo. The fact that this later wqo is finite plays no role. Indeed, members of the age of
an almost multichainable structure can be coded by words over a finite alphabet, labelled members can be coded also by words, the conclusion follows from Higman's theorem on words.

With Theorem \ref{thm:bounds}, we have:

\begin{theorem}
If the signature is bounded, the cardinality of bounds of the age of an almost multichainable structure is bounded.
\end{theorem}

We extend Proposition \ref{multichainablewqo}.

\begin{proposition}\label{multichainablebqo}If $\mathcal C$ is the age of a almost multichainable structure  then the collection  $\tilde{\mathcal C}_{\leq\omega}$ of countable structures whose ages are included in $\mathcal C$ is bqo and  in fact hereditary bqo.
\end{proposition}
\begin{proof} The proof of Proposition \ref{multichainablewqo} given in \cite {pouzet2006}
 consists to interpret members of the age by words over a finite alphabet and apply Higman's Theorem on words. Here, using the same alphabet, we interpret the countable structures whose ages are included in $\mathcal C$ by countable chains labelled by this alphabet. Since the alphabet is bqo,  Laver's Theorem implies that this collection of labelled chains is bqo. Now, if we label these labelled chains by a bqo,  an other application of Laver's Theorem  yields that the resulting class is bqo.  Here are the details.
 We asserts  that $\tilde{\mathcal C}_{\leq\omega}$ is the image of the  product $\mathcal P$ of a finite poset and the class $\mathcal S_{\leq \omega}$ labelled by a finite set $A$ by a map  $\varphi$ preserving the order. Any finite poset being  bqo, $\mathcal S_{\leq \omega}.A$ is bqo via  Laver's Theorem,  the product $\mathcal P$ is bqo.  and its image too,  hence $\tilde{\mathcal C}_{\leq\omega}$ is bqo. If we label members of $\tilde{\mathcal C}_{\leq\omega}$ by a bqo, say $Q$,  this amounts to label the members of $\mathcal P$ by $Q$. An other application of Laver's theorem  yields that $\mathcal P.Q$ is bqo, hence $\tilde{\mathcal C}_{\leq\omega}. Q$ is bqo as required.  To prove our assertion, let $R:=: (V, (\rho_i)_{i\in I})$ be an almost multichainable structure as defined above such that  $\age (R)= \mathcal C$. The set   $V$ is the disjoint union of a finite  set $F$ and a set  of the form $L\times K$ where  $K$ is a finite set and  there is a linear  order on $L$  satisfying condition $\bullet$ above.  We code $R$ by an \emph{almost multichain} $S:= (V, \leq, \rho_L, \rho_K, a_1, \dots a_{\ell})$ where  the  elements $a_1, \dots, a_k$ are constants defining $F$ that is $\{a_1, \dots, a_k\}=F$,  the relations $\rho_L$ and $\rho_K$ are  equivalence relations on $V$ such that all the elements of $F$ are inequivalent and inequivalent to those of $L\times K$, furthermore two elements   $(x,y)$ and $(x',y')$ of $L\times K$ are equivalent with respect to   $\rho_K$   if $y=y'$, respectively $\rho_L$ if $x=x'$;  the relation $\leq $ is an order on $V$ for which all the elements of $F$ are minimal and below all elements of $L\times K$, all blocks of $\rho_L$ are antichains for the ordering and the family of these blocks is totally ordered by $\leq$. Let $\mathcal D:= \age( S)$ and $\tilde {\mathcal D}_{\leq \omega} $ be the collection of  countable structures whose age is included in $\mathcal D$. We may apply Compactness Theorem of First Order Logic to this elaborate construction. It yields that for every member $R'$ of $\mathcal C= \age (R)$ there is some $S'$ on the same domain such that $\age (S') \subseteq \age (S)= \mathcal D$ and every local isomorphism from $S'$ to $S$ yields a local isomorphism  from  $R'$ to $R$ (beware, $S'$ is not necessarily almost multichainable).
Thus  $\tilde{\mathcal D}_{\leq\omega}$ is transformed in  $\tilde{\mathcal C}_{\leq\omega}$ in an order preserving way. Hence, if $Q$ is a bqo, in order to prove that $\tilde{\mathcal C}_{\leq\omega}. Q$ is  bqo,  it suffices to prove that $\tilde{\mathcal D}_{\leq\omega}. Q$. is bqo. For this new step, let $A$ be the alphabet  whose letters are the non-empty subsets of $K$ (that is
$A := \powerset (K ) \setminus \{\emptyset \}$. We order $A$ by inclusion and
extends this order to  the collection  $\mathcal S_{\leq \omega}. A$ of countable chains  labelled by the elements of $A$.  Order $\powerset (F)$  by inclusion. Then $\tilde{\mathcal D}_{\leq\omega}$ is the image of  the direct product $\mathcal P:= \powerset (F) \times \mathcal S_{\leq \omega}. A$ (thus, via Laver's Theorem, $\tilde{\mathcal D}_{\leq\omega}$ is bqo).  Then, replace  $A$ by the set $A'$ of $f: X\rightarrow  Q$ such that $X\in A$ in this direct product. An other application of Laver's theorem  yields that $\tilde{\mathcal D}_{\leq\omega}. Q$ is  bqo. \end{proof}

Since $D(m)$ is multichainable, we obtain:

\begin{corollary} \label{multichainablebqo2} The collection of countable posets  whose age  in  included in $\age (D(m))$ is hereditary bqo.
\end{corollary}
\subsection{A proof of Theorem \ref{thm:1}}\label{section:proofthm:1}

We prove that the following implications hold: $(i)\Rightarrow (ii)\Rightarrow (iii)\Rightarrow  (iv)\Rightarrow (i)$.

We prove first  implications  $(iv)\Rightarrow (i)$ and $(iii)\Rightarrow (iv)$. They are almost immediate.
We prove implication $\neg (i)\Rightarrow \neg (iv)$. Suppose that there are infinitely many integers $k$ such that $\mathrm P_k \in \ainc \langle \mathcal C \rangle$. For each $k$, let $\widehat {\mathrm P}_k:=(\mathrm P_k, a_k, b_k)$ where $a_k$ and $b_k$ are two constants denoting the end vertices of $\mathrm P_k$. The $\widehat {\mathrm P}_k$'s  form an antichain with respect to embeddability. Indeed, suppose otherwise that some $\widehat {\mathrm P}_k$ embeds into some $\widehat {\mathrm P}_k'$. Then first, $k<k'$; next, since the  distance between $a_k$ and $b_k$ in $\mathrm P_k$ is $k-1$,  the distance between their images, namely $a_{k'}$ and $b_{k'}$,  is at most $k-1$ but in $\mathrm P_{k'}$ it is $k'-1$; This yields a contradiction. Since the  class of  members of $\ainc \langle \mathcal C \rangle$ labelled with two constants contain an infinite antichain, it  cannot be wqo,   a fortiori the  class of  members of $\mathcal C$ labelled with two constants cannot be wqo. Now, we prove implication $(iii)\Rightarrow (iv)$. The class $\mathcal C_{< \omega}$ is a subclass of $\mathcal C$, hence it inherits of the properties of $\mathcal C$.  Since the two element chain $2:= \{0, 1\}$, with $0<1$, is bqo, the class  $\mathcal  C_{<\omega}.2$ of members of $\mathcal C_{<\omega}$ labelled by the $2$-element chain $2$ is bqo hence wqo. We prove implication $(ii)\Rightarrow (iii)$. The class $\mathcal C$ is a subclass of the collection of countable posets whose age is included in $\age (D(m))$. According to  Corollary  \ref{multichainablebqo2} this latter class is hereditarily bqo. Hence, $\mathcal C$ is hereditarily bqo. We prove that  $(i)\Rightarrow (ii)$.  Suppose that for some nonnegative integer $k$, $\mathrm P_{k}$ does not belong to $\ainc \langle \mathcal C \rangle$. Then, no member of $\ainc \langle \mathcal C \rangle$ contains an induced  path $\mathrm P_k$. In particular, the diameter of every connected member of $\ainc \langle \mathcal C \rangle$ is at most $k-1$. According to Proposition \ref{bounded-diameter}, each countable poset $P$ of width at most two,  whose incomparability graph is connected and has  diameter at most $k$,  is embeddable in $D(m)$ for some nonnegative integer $m$.  Finite sums of finite members of $\mathcal C$ are embeddable in $D(m)$ hence $(ii)$ holds.  \hfill $\Box$
%

\subsection {A proof of Theorem \ref{thm:2}}

\noindent $(i) \Rightarrow (v)$. Since $\mathcal C_{<\omega}\subseteq \mathcal C$ and $\mathcal C$ is bqo, $\mathcal C_{<\omega}$ is bqo hence wqo.  \\
$(v)\Rightarrow (ii)$. Since $(\mathrm {DF}_{n})_{n\geq 1}$ is an antichain, the fact that $\mathcal C_{<\omega}$ contains no infinite antichain implies that it  contains at most finitely many members of this sequence and thus $(ii)$ holds.  \\
$(ii) \Rightarrow (iii)$. According to Lemma \ref{lem3:doublefork} of  subsection \ref{subsection:diameter}, if   there are vertices of degree at least three at arbitrarily large distance, there is an infinite antichain of double-ended forks.\\
$(iii) \Rightarrow (iv)$. According to Lemma \ref{lem2:degreethree} of  subsection \ref {subsection:diameter},  given $G$ in $\mathcal C$ such that the diameter of the set of vertices of degree at least $3$ is at most $r$ there exists a connected induced subgraph $K$ of $G$ with diameter at most $r+4$ such that $G$ is the union of $K$ and two paths attached to two vertices of $H$. The proof of Lemma \ref{lem2:degreethree} tells us that one of the paths is below the other. Hence if $G= \ainc(P)$, $P$ is as claimed.   \\
$(iv) \Rightarrow (i)$.  Let $\mathcal D$ be the class of $P\in \mathcal C$ such that $\ainc(P)$ is connected. Every $P\in \mathcal C$ is a lexicographic sum over a countable chain of members of $\mathcal D$. If $\mathcal D$ is bqo then by Laver's theorem the class of countable chains labelled by  $\mathcal D$ is b.q.o. Let $\mathcal D ^{(k)}$ be the subclass of $\mathcal D$ made  of posets $K$ such that $\delta (\ainc(K))\leq k$. Since every member of this class embeds in $D(m(k))$, the class $\mathcal D ^{(k)}$ is bqo and in fact hereditary bqo. Let $L(\mathcal D ^{(k)})$ be the class of members $K\in  \mathcal D ^{(k)}$ with two distinguished  vertices $a, b$,  labeled by elements $n_a$ and $n_b$ of the chain $\NN\cup \{\infty\}$. This class is a subclass of $\mathcal D ^{(k)}$ labelled by a bqo hence it is bqo. To each member  $(K, a, n_a, b, n_b)$ of $L(\mathcal D ^{(k)})$ associate the poset  $P$ made of $K$ and  the union two orientations of the complement of  a path $C_a$ of length $n_a$ and a path $C_b$ of length $n_b$ so that every vertex of
$P$ every vertex  of $C_{a} \setminus \{a\}$ is below every vertex of $K\setminus \{a\}$, and every vertex  of $C_{b} \setminus \{b\}$ is above  every vertex  of $K\setminus \{b\}$. In this association, the embeddability is preserved, hence the image is bqo.
According to $(iv)$,  each member of $\mathcal D$ is attained.  Hence $\mathcal D$ is bqo.
\hfill $\Box$

\begin{comments}\label{comment:lozin}
Theorem 7 page 520 of \cite{lozin-mayhill} states that for any fixed $k$, the set of finite bipartite permutation
graphs that do not embed a $DF_n$ for all $n\geq k$ is well-quasi-ordered.  In their proof, the authors stated that: "The proof of this theorem is similar to that of Theorem 5, so we reduce it to a sketch". By checking the proof of Theorem 5 \cite{lozin-mayhill} carefully we have found that the statement "We will show that in this case $G_i$ is an induced subgraph of $G_j$" is inaccurate. Indeed, if $G_i$ is $\Inc(P)$ and $G_j$ is $\Inc(Q)$ from Figure \ref{fig:counterexample-Lozin} one can easily see that $G_i$ and $G_j$ are unit interval graphs and do not embed in each other. Yet, by setting $L_i=\{1,2\}$, $M_i=\{2,3,4\}$, $R_i=\{3,5\}$, $L_j=\{a,b\}$, $M_j=\{b,c,d,e\}$, $R_j=\{e,f\}$ one can easily see that $L_i$ is an induced subgraph of $L_j$ , $M_i$ is an induced subgraph of $M_j$ and $R_i$ is an induced subgraph of $R_j$. Higman's Theorem is not enough to yield the desired conclusion here. The desired  conclusion follows because the set of "middle graphs" (i.e the $M_i$'s) is hereditary wqo.   This  is our Theorem \ref{thm:1}. In fact, in \cite{atmitas-lozin} the authors,   considering the notion of hereditary wqo (called labelled  wqo),  prove an analog to our Proposition \ref {multichainablewqo} (see their Theorem 4) leading to Theorem \ref{thm:1}.
\end{comments}
\begin{figure}[!ht]
\begin{center}
\includegraphics[width=180pt]{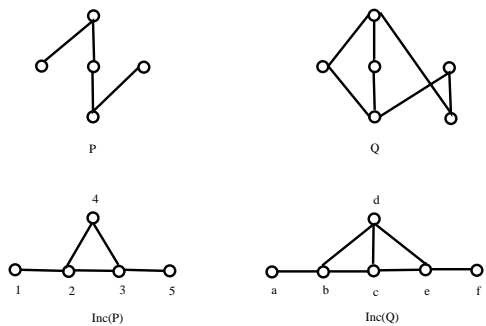}
\end{center}
\caption{Posets $P$ and $Q$ and their corresponding incomparability graphs.}
\label{fig:counterexample-Lozin}
\end{figure}

\section{Posets of width two:   a test class for equivalent versions of wqo.}\label{section:testclass}
 We give equivalent conditions ensuring that a hereditary class of finite structures is wqo (Proposition \ref{prop:well-founded}). We indicate that for ages, these equivalence remain conjectures. But we shown that they hold for  ages of posets of width two and the associated classes of  graphs and bichains. We conclude by a glimpse at hereditary classes of countable structures.

 \subsection{Comparison of finite  structures}
 \begin{proposition}\label{prop:well-founded}
Let $\mathcal{C}$ be a hereditary class of finite structures. Then the following properties are equivalent:
\begin{enumerate}[(a)]
\item $\mathcal{C}$ contains an  infinite antichain;
  \item $\mathcal{C}$ contains  uncountably many hereditary subclasses.
  \item The collection $\mathbf I(\mathcal C)$ of hereditary subclasses of $\mathcal{C}$ ordered by set inclusion  is not well-founded.
  \item There exists a collection $\mathcal D$ of hereditary subclasses of $\mathcal{C}$ that is  isomorphic to $\powerset(\NN)$,
      the power set of $\NN$, the sets $\mathcal D$ and $\powerset(\NN)$ being ordered by set inclusion.
\end{enumerate}
\end{proposition}

The proof is straightforward: Suppose that Item $(a)$ holds. Let  $\mathcal A$ be an infinite antichain in $\mathcal C$;  let $(G_n)_{n \in \NN}$ be an enumeration of $\mathcal A$; the map $\varphi: \powerset (\NN) \rightarrow \mathbf I(\mathcal A)$ defined by setting $\varphi(I):= \{G:G \leq G_n \; \text{for some}\; n\in I\}$ for each subset $I$ of $\NN$ is an embedding and Item $(d)$ holds. If Item $(d)$ holds then since $\powerset(\NN)$ is both uncountable and not  well-founded, Item $(b)$  and Item $(c)$ hold. If Item $(b)$ or Item $(c)$ holds  then there is an infinite antichain in  $\mathcal C$.  Otherwise, in the first case,  since $\mathcal C$ is well-founded then each hereditary subclass has finitely many bounds belonging to $\mathcal C$; since each hereditary subclass is determined by its bounds, and since $\mathcal C$ is finite or countable, the number of these classes is at most countable. In the second case, since there is a strictly descending sequence $(\mathcal C_n)_{n\in \NN}$ of hereditary subclasses of $\mathcal C$,  pick  $G_n\in \mathcal C_{n}\setminus \mathcal C_{n+1}$. Since in $\mathcal C$ there is no infinite descending sequence (in fact, for each $G\in \mathcal C$ there are only finitely many $G'$ that embed into $G$),  we can extract from this sequence an infinite  subsequence forming an infinite antichain (this is essentially Ramsey's theorem on pairs).

It is a long standing unsolved question to know whether the previous equivalences are still true if we replace $\mathcal{C}$ by an age and the hereditary subclasses of $\mathcal C$ by ages. The only fact known up to now  is that if $\Id(\mathcal C)$,  the collection of ages included in $\mathcal C$, ordered by inclusion,  is countable, or the chains included in that set are countable,  then $\Id(\mathcal C)$ is well-founded (see \cite{pouzet78bis} and \cite{pouzet79}). It must be noticed that the fact that  $\powerset(\NN)$  embeds as a poset in $\Id(\mathcal C)$, is equivalent to the fact that $[\omega]^{<\omega}$, the collection of finite subsets of $\omega$, ordered by inclusion, embeds into $\mathcal {C}$ (this is a general fact about posets, see Theorem 1.1 p. 211 and Proposition 2.2 p. 215 of \cite{chakir-pouzet} for a proof). Since hereditary classes are countable, they embed in $[\omega]^{<\omega}$ (to each $\tau$ belonging to a hereditary class $\mathcal {C}$ associate the finite set $\downarrow \tau=\{\tau'\in \mathcal C: \tau'\leq \tau\}$). Hence we have the following dichotomy result:

\begin{lemma}
 For every age $\mathcal{C}$, either  $\mathcal{C}$ does not embed $[\omega]^{<\omega}$ or $\mathcal {C}$ is \emph{equimorphic} to $[\omega]^{<\omega}$, in the sense that each  embeds in the other.
\end{lemma}

It is a tempting (and old) conjecture that the fact that $\mathcal {C}$ does not embed $[\omega]^{<\omega}$ is equivalent to the fact that $\mathcal C$ contains no infinite antichain, that is $\mathcal C$ is wqo.

In the case of ages consisting of finite posets of width two the answer is positive. More is true, the full result,  which is a corollary of Theorem \ref{thm:3}, reads as follows.

\begin{corollary}\label{cor:subclasses} Let $B$ be a bichain, $P:= o(B)$ and $G:= \ainc (P)$.  If $\mathcal C$ is  the age of one of these structures and if   $\Id(\mathcal C)$ is the collection of  ages
included into $\mathcal C$,  then  the following propositions are equivalent.
\begin{enumerate}[(i)]
  \item $\mathcal{C}$ is wqo;
\item $\Id(\mathcal{C})$ is countable;
  \item $\Id(\mathcal{C})$ is well-founded;
  \item $\powerset(\NN)$ does not embed in $\Id(\mathcal{C})$;
  \item $[\omega]^{<\omega}$ does not embed in $\mathcal{C}$.
\end{enumerate}
\end{corollary}
\begin{proof} If $\mathcal C$ is the age of an arbitrary structure,  the implications $(i)\Rightarrow (ii)  \Rightarrow (iii) \Rightarrow (iv) \Rightarrow (v)$ hold. The only nontrivial implication $(ii\Rightarrow (iii)$ is Th\'eor\`eme V-4 page 74 of \cite{pouzet78bis}.  To conclude, it suffices to prove that $\neg (i) \Rightarrow \neg (v)$. Suppose that $\mathcal C=\age (G)$ and $(i)$ does not hold. Apply Theorem \ref{thm:3}. Since $\mathcal C$ contains an infinite antichain,  there is a infinite subset $I$ of $\NN$ such that $\age
(\bigoplus_{n\in I}DF_n)\subseteq \mathcal{C}$. For every finite subset $F$ of
$I$, set $\varphi(F):=  \bigoplus_{n\in F}DF_n$. This
leads to an embedding of $[I]^{<\omega}$, the set of finite subsets of $I$, ordered by inclusion, in  $\Id(\mathcal{C})$, hence
an embedding of $[\omega]^{<\omega}$ in  $\Id(\mathcal{C})$, since, as posets, $[I]^{<\omega}$ and
$[\omega]^{<\omega}$ are isomorphic. Hence $(v)$ does not hold. According to Theorem \ref{thm:bqo-poset-bichain}, if $\age (B)$  or $\age (P)$ are not wqo then $\age (G)$ is not wqo. To the infinite  antichain of $DF_n$ with $n\in I$ as above corresponds an antichain in $\age (B)$ and in $\age (P)$ and it yields and embedding of $[\omega]^{<\omega}$ in $\age (B)$ and in $\age ( P)$.  Hence $\neg (i) \Rightarrow \neg (v)$ holds for $\mathcal C= \age ( B)$ as well $\age (P)$.  \end{proof}

\subsection{Comparison of countable structures}
Given a hereditary class $\mathcal C$ of finite structures, it is natural to look at the class $\tilde{\mathcal C}$ of structures, e.g.  graphs $G$,   whose age $\age (G)$ is a subset of $\mathcal C$  and to ask whether  the properties of the embeddability on $\mathcal  C$ have an influence on embedabbility on $\tilde{\mathcal C}$. For an example, if $\mathcal C$ is wqo, does  $\tilde{\mathcal C}$ wqo? This very innocent question goes quite far. If we look at chains (linearly ordered sets), the finite chains are ordered according to their size and thus form a chain of order type $\omega$. There are two countable  chains, namely $\omega$ and its dual $\omega^*$ and each infinite chain embeds at least of these two. The chain $\eta$ of rational numbers embeds all countable chains. There are descending sequence of uncountable chains. Fra\"{\i}ss\'e asked in 1954 if there were descending sequences of countable chains, and if fact if the collection of countable chains, quasi ordered by embeddability,  is wqo. These famous conjectures were solved by Laver (1971) \cite{laver}, using the notion of better-quasi-ordering (bqo). Thomass\'e \cite{thomasse} proved that  the class of countable cographs is wqo (and in fact bqo) after Damaschke \cite{damaschke} shown  that the class of finite cographs  is wqo. An extension to binary structures was obtained independently by Delhomm\'e \cite{delhomme1} and McKay \cite{mckay, mckay1}:

\begin{theorem}\label{bqo-hereditary2} If a hereditary class $\mathcal C$ of finite binary structures of finite arity contains only finitely many prime members then the class of countable  structures whose age is included into $\mathcal C$ is bqo.
\end{theorem}

At present, it is not known if a hereditary class of finite graphs which is wqo is bqo. There are examples of hereditary classes $\mathcal C$  of finite graphs which are  wqo and in fact bqo,  but such that the class $\tilde {\mathcal C}_{\leq \omega}$ of countable members of $\tilde {\mathcal C}$ is not wqo. Examples can be made with uniformly recurrent sequences (see Comments \ref{comments-uniformly} in Section \ref{section:minimality}); they  are not hereditary wqo. This leading to the conjecture that $\tilde {\mathcal C}_{\leq \omega}$ is wqo provided that ${\mathcal C}$  is hereditary wqo.

With Theorem \ref{thm:bqo-poset-bichain} proved in Section \ref{section:bichains}, we know that if $\mathcal C$ is a  hereditary class  of finite  bipartite permutation graphs,  then   $\mathcal C$  is wqo if and only if it is bqo. Moreover,
$\tilde {\mathcal C}_{<\omega}$ is bqo if $\mathcal C$  is wqo.

\begin{problem}\label{problem6}  Does this hold  for hereditary classes of finite permutation graphs?
\end{problem}

Finite cographs are permutation graphs. Due to Thomass\'e's result, our problem has a positive answer  for hereditary subclasses of the class of finite cographs.

\section{Embeddability between posets, bichains and graphs. A  proof of Theorem~\ref{thm:bqo-poset-bichain}}\label{section:bichains}
In this section,  we prove first that embeddability between certain posets of dimension two is equivalent to embeddability between the corresponding incomparability graphs and bichains (Theorems \ref{thm:embed-poset-graph} and \ref{thm:embed-bichain-poset}). Then we prove Theorem~\ref{thm:bqo-poset-bichain}.

To that end, we recall the notion of primality, two basic results about posets and bichains (Theorems  \ref{kelly} and \ref{zaguia}) and we extend to infinite posets a previous result of El-Zahar and Sauer \cite{el-zahar-sauer} characterizing the finite posets of dimension two which are uniquely realizable.

We start with incomparability graphs and   then deal with bichains.

\subsection {Primality and embeddability between posets and graphs}
We refer to the book \cite{ehrenfeucht1} and the founding work of Fra\"{\i}ss\'e and Gallai \cite{fraisse84, gallai}.  We recall the notions of modules and strong modules. This later notion was introduced by  Gallai (1967) who  showed that a finite poset and its comparability graph have the same strong modules, from which it follows that one is prime if and only if the other one is prime. He also proved that if a finite graph is prime then the only other poset with the same comparability graph is the dual poset. These results were extended to infinite comparability graphs by Kelly \cite{kelly85}.

Let $\rho$ be a binary relation on a set $V$, i.e. a subset of $V\times
V$. A {\it module} (or {\it an interval}) for $\rho$ is a subset $A$ of $V$ such that
\[ ((v,a)\in \rho \Rightarrow (v,a') \in \rho) \; {\rm and } \;((a,v) \in \rho \Rightarrow (a',v)\in \rho) \]
hold for all $a,a'$ in $A$ and $v$ in $V\setminus A$. \\
A \emph{module} of a binary structure $R:=(V,(\rho_{i})_{i\in I})$, where each $\rho_i$ is a binary relation on $V$, is a module of $\rho_i$
for all $i\in I$. The empty set, the singletons in $V$ and the whole set $V$ are modules of $R$ and are said to be {\it trivial}. A module is \emph{proper} if it is distinct of $V$. A binary relational structure is {\it prime} (or {\it indecomposable}) if it has no nontrivial modules. Incidentally, a relational structure on at most two elements is prime.  A module is \emph{strong} if it is nonempty and it overlaps no module \cite{gallai}. The singletons and the whole set are strong and are called \emph{trivial}.  A   \emph{component} of $R$ is a  maximal (with respect to set inclusion) proper strong module of $R$.

We recall Gallai's decomposition theorem: \emph{if $R$ is finite,  then $R$ is the lexicographic sum of the restrictions of $R$ to its components (hence, the strong modules of the index set of this sum are trivial)}. More importantly, \emph{if $R$ is a poset or a graph, all  strong modules  of $R$ are trivial if and only if $R$ is either prime or is a chain, a complete graph or has no edges} (\cite{gallai} for the finite case, \cite{ehrenfeucht} for the infinite case).

The following  result (\cite{gallai} for finite graphs, \cite{kelly85} for infinite graphs) plays a crucial role in our investigation.

\begin{lemma}\label{lem:kelly} A poset $P$ and its comparability graph $\comp(P)$ have the same strong modules. \end{lemma}

A consequence is the result of Gallai-Kelly:
\begin{theorem} \label{kelly} A poset $P$ is prime if and only if $\ainc(P)$ is prime if and only if
$\comp(P)$ is prime. Moreover, if $\comp(P)$ is prime,  then it has exactly two transitive orientations, namely $P$ and $P^*$.
\end{theorem}

\begin{corollary}\label{proposition:embed-poset-graph-prime}
Let $P$ and $Q$ be two  posets. If $P$ is prime,  then the following propositions are equivalent.
\begin{enumerate}[$(i)$]
  \item $P$ embeds into $Q$ or $Q^*$.
  \item $Comp(P)$ embeds into $Comp(Q)$.
  \item $\Inc(P)$ embeds into $\Inc(Q)$.
\end{enumerate}
\end{corollary}

We do not know how to extend this result to non prime posets in general.  We do have an extension when  the only modules distinct  from $P$ are finite and totally ordered. This gives our first embeddability result.

\begin{theorem}\label{thm:embed-poset-graph}Let $P=(V,\leq_P)$ and $Q=(U,\leq_Q)$ be two posets such that  the only modules of $P$ distinct from $P$ are finite and totally ordered. The following propositions are equivalent.
\begin{enumerate}[$(i)$]
  \item $P$ embeds into $Q$ or $Q^*$.
  \item $Comp(P)$ embeds into $Comp(Q)$.
  \item $\Inc(P)$ embeds into $\Inc(Q)$.
\end{enumerate}
\end{theorem}

We could deduce this result from Gallai's decomposition theorem. We prefer to deduce it from properties of  totally ordered modules, that we will use for the study of bichains.

\begin{lemma}\label{lem:module-chain} Let $P$ be a poset. Then:
\begin{enumerate}[$(1)$]
\item The union $A\cup B$ of two totally ordered modules $A$ and $B$ of $P$ such that  $A\cap B\neq \emptyset$   is a totally ordered module of $P$;
\item Every totally ordered module of $P$ is contained in a totally ordered module of $P$, maximal with respect to inclusion.
\end{enumerate}
\end{lemma}

\begin{proof} $(1)$ As it is well known \cite{ille} the union of two modules with a nonempty intersection is a module.
We check that $A\cup B$ is totally ordered.  Let $a\in A$ and $b\in B$. Since $A\cap B$ is a module and is a chain, $a$ is comparable to all elements of $A\cap B$. From the fact that $B$ is a module we infer that $a$ is comparable to all elements of $B$ and hence $a$ is comparable to $b$. This shows that the order induced by $P$ on $A\cup B$ is total and we are done.\\
$(2)$ This second assumption of the lemma follows readily from the previous one and Zorn's Lemma.
\end{proof}

Let $P$ be a poset.  Let $x, y \in P$. We set $x \equiv_P y$ if there exists a totally ordered module $A$ of $P$ containing $x$ and $y$. Clearly, $\equiv_P$ is reflexive and symmetric. It follows from Lemma \ref{lem:module-chain} that $\equiv_P$ is also transitive. Hence, $\equiv_P$ is an equivalence relation. The equivalence class of each element is a totally ordered module. We denote by $P/\equiv_P$ the quotient poset and let $\varphi:P \rightarrow P/\equiv_P$ be the quotient map.

\begin{lemma}\label{lem:gallaiquotient}Let $P$ be a poset. Then:
\begin{enumerate}[$(1)$]
\item $P$ is the lexicographical sum $\sum_{i\in D} C_i$ of chains $C_i$ indexed by a poset $D$ in which no nontrivial module is totally ordered.
\item If every module of $P$ distinct of $P$ is totally ordered then $D$ is prime.
\end{enumerate}
\end{lemma}

\begin{proof}
$(1)$  $P$ is the lexicographical sum of its equivalence classes indexed by $P/\equiv_P$. If $A$ is a totally ordered  module of $P/\equiv_P$ then  its inverse image by $\varphi$ is a  totally ordered module of $P$. Due to the maximality of equivalence classes, $\vert A\vert \leq 1$. So we may set   $D:=P/\equiv_P$.

$(2)$ Suppose that $D$ is not prime. Let $A$ be a proper module of $D$. Then $\cup\{C_i : i\in A\}$ is a proper module of $P$, hence is totally ordered. But then $A$ was totally ordered, contradicting the fact that no proper module of $D$ is totally ordered.
\end{proof}

\noindent{\bf Proof of Theorem \ref {thm:embed-poset-graph}.}
The equivalence $(ii) \Leftrightarrow (iii)$ and implication $(i) \Rightarrow (ii)$ are trivial and  true with no conditions on $P$ and $Q$.\\
$(ii) \Rightarrow (i)$ Let $f$ be an embedding of $\comp(P)$ into $\comp(Q)$, $V':= f(V)$, $P':= (V', \leq')$ where $\leq'$  is the image on $V'$ of the order $\leq$ on $P$ and $Q':= Q_{\restriction V'}$.  By Lemma \ref{lem:gallaiquotient}, the poset $P'$ is a lexicographic sum of chains $C_i$ with domains $A_i$ indexed by a prime poset $D'$. Hence $\comp(P)$ is a lexicographical sum of the cliques $\comp(C_i)$  indexed by $\comp (D')$. By Theorem \ref{kelly},  $\comp (D')$ is prime since $D'$ is prime. Since  $\comp(P')=\comp(Q')$, $Q'$ is a lexicographical sum of chains $C'_i$ with domain $A_i$ indexed by an orientation $D''$ of $\comp(D')$. By Theorem \ref{kelly}, $D''=D'$ or $D''= D{'*}$. Since $P$ is finite each chain $C_i$ is isomorphic to $C'_i$ as well as its dual. Hence $P'$ is isomorphic to $Q'$ or to  $Q^{'*}$. Hence $P$ embeds into $Q$  or into $Q^*$ and we are done. \hfill $\Box$

It should be mentioned that Theorem \ref{thm:embed-poset-graph} as stated becomes false if the condition that the modules of $P$ are finite chains is dropped. Indeed,  two infinite chains have the same comparability graphs but, in general, do not embed in each other.

We now specialize the above results to our situation: aka posets of width two. In particular, we will show that if  a poset has width at most two and has a connected incomparability graph, then the only possible nontrivial modules are totally ordered.
Note that if  $P$ is a poset such that every module distinct from $P$ is a chain, then the nontrivial strong modules of $P$ are the maximal nontrivial modules of $P$. This leads to an other proof of Theorem \ref{thm:embed-poset-graph}.

\begin{lemma}\label{lem:module1}Let $P:=(V,\leq)$ be a poset and $I$ be a module in $P$. If $I$ contains a maximal antichain of $P$, then $\Inc(P)$ is not connected.
\end{lemma}
\begin{proof} Let $I$ be a module in $P$ and let $A$ be a maximal antichain of $P$ such that $A\subseteq I$. Since $A$ is maximal with respect to  inclusion,  every element of $P$ is comparable to at least one element of $A$. Hence, every element not in $I$ is comparable to at least one element of $A$ and therefore is comparable to all elements of $I$ since $I$ is a module. Since a module is convex we infer that every element not in $I$ is either larger than all elements of $I$ or is smaller than all elements of $I$. This proves that $I$ is a nontrivial component of $\ainc(P)$. The conclusion of the lemma follows.
\end{proof}

If $P$ has width at most two, maximal antichains have size at most two, hence:

\begin{corollary}\label{cor:module1}Let $P:=(V,\leq)$ be a poset of width two. If  $\ainc(P)$ is connected then all proper  modules of $P$ are chains.
\end{corollary}

If a poset  has width at most two, its incomparability graph is bipartite. A characterisation of bipartite prime graph is easy:

\begin{proposition}\label{primebipartite} A bipartite graph is prime if and only if it is connected and distinct vertices have distinct
neighborhoods.
\end{proposition}
\begin{proof}
 The proof is easy and left to the reader.
\end{proof}

\subsection{Embeddability between bichains}

A \emph{bichain} is a relational structure $B := (V, (\leq_1,\leq_2))$ where $\leq_1$ and $\leq_2$ are two total orders on the same set $V$. The bichain \emph{induced} by $B$ on a subset $A$ of $V$, also called the \emph{restriction} of $B$ to $A$, is the bichain
$B_{\restriction A} := (A, ({\leq_1}_{\restriction A}, {\leq_2}_{\restriction A}))$ where $\leq_{1\restriction A}$ and ${\leq_2}_{\restriction A}$ are the orders induced by $\leq_1$ and $\leq_2$ on $A$. If $B := (V, (\leq_1, \leq_2))$ is a bichain, its  \emph{transpose} is $B^{t}:=(V, (\leq_2, \leq_1))$, its dual is  $B^*:= (V, (\leq_1^*,\leq_2^*))$. Clearly, ${(B^{t})}^t=B$ and ${(B^{*})}^*=B$. Furthermore, ${(B^{t})}^*={(B^{*})}^t$. It follows that starting from the bichain $B$ and applying the operations transpose and dual we may possibly generate at most three new bichains.

Let $P:= (V, \leq)$ be a poset; a \emph{realizer} of $P$ is a set $\mathcal E$ of linear orders on $V$ whose intersection is the order on $P$. We also say that the set $\{(V, \preceq):\;  \preceq \in \mathcal E\}$ is a realizer of $P$.  The order on  $P$, or $P$ as well,  is \emph{uniquely realizable} if it has only one realizer with at most two elements.  In  particular $P$ has dimension at most two.

The relevance of this notion is in the following obvious lemma:

\begin{lemma}A poset $P$ of dimension two has a unique realizer if and only if  there are exactly two bichains $B$ and $B'$ such that $o(B)=o(B')=P$. In which case $B'=B^t$.
\end{lemma}

 In this case, we denote by $B_P$ one of these two bichains.

 Our first embeddability result is this.

\begin{theorem}\label{thm:embed-bichain-poset}Let $P$,  $P'$ be two posets. If   $P$ is uniquely realizable,  then $P$ embeds into $P'$ if and only if $B_P$ embeds into $B_{P'}$ or into  $B^t_{P'}$.
\end{theorem}
\begin{proof}
 Suppose that $B_P$ embeds into $B_{P'}$ or in $B^t_{P'}$.  Then $o(B_P)$ embeds in $o(B_{P'})$ or in $o(B^{t}_{P'})$. Since $P=o(B_P)$ and $P'=o(B_{P'})= o(B^{t}_{P'})$, $P$ embeds in $P'$.
Conversely, suppose that $P$ embeds into $P'$.  Let $\varphi$ be such an embedding and $X$ be the image of $\varphi$. Then $B(P')$ induces a realizer of $P'_{\restriction X}$, hence a realizer of $P$. Since $P$ is uniquely realizable, $B$ and $B^t$ are the only bichains such that $o(B)=o(B^t)=P$ hence $\varphi$ is an embedding of $B$  into $B'$ or into $B'^t$ as required.
\end{proof}

We characterize the uniquely realizable posets. We extend the result  of El-Zahar and Sauer to arbitrary  two-dimensional posets.

\begin{theorem}\label{thm:elzahar-sauer} A   poset $P$ of dimension two is uniquely realizable if and only if at most one connected component of $\ainc(P)$ is nontrivial and the only proper modules of the poset induced on this component are totally ordered.
\end{theorem}

With  Lemma \ref{lem:gallaiquotient} we obtain that   \emph{a poset $P$  is uniquely realizable if and only if $P$  is the lexicographical sum $Q^{-} + Q+ Q^{+}$ where $Q^{-}$ and $Q^{+}$ are chains and $Q$ is a lexicographical sum of chains indexed by a two-dimensional prime poset}.

As a consequence:
\begin{corollary}\label{cor:elzahar-sauer}  Prime  posets of dimension two are uniquely realizable.
\end{corollary}

In fact, the proof of Theorem \ref{thm:elzahar-sauer} is based on the corollary. The proof of the corollary uses the finite case and 	a  result of Ille (\cite{ille} for binary relations, \cite{delhomme} for binary structures).

\begin{theorem}\label{thm:ille}A binary structure  $R:=(V,(\rho_{i})_{i\in I})$ is prime if and only if for every finite set $F\subset V$ there exists a  finite subset  $F'$ of $V$ such that $R_{\restriction F'}$ is prime. \end{theorem}

Indeed, let $P:= (V, \leq)$ be a prime poset of dimension two.  In order to prove that  there is just one   realizer, suppose that there are two. Observe that  their restrictions to some subset with at most four elements $F$ of $V$must be  distinct. Since $P$ is prime, Theorem \ref{thm:ille} asserts that  there is some finite subset $F'$ of $V$ containing $F$ such that $P_{\restriction F'}$ is prime. According to Theorem 1 page 240  of \cite{el-zahar-sauer}, $P_{\restriction F'}$ is uniquely realizable. This contradicts our hypothesis.

We recall the following result  stated in \cite{zaguia98} for finite sets; the proof given holds without that restriction.

\begin{theorem} \label{zaguia}  Let $B:= (V, (\leq_1, \leq_2))$ be a bichain.  Then the poset $o(B):= (V, \leq_1\cap \leq_2)$ is prime if and only if the bichain $B$ is prime.
\end{theorem}

Putting together Corollary \ref{cor:elzahar-sauer}  and Theorem \ref{zaguia} we obtain:
\begin{corollary}\label{cor:elzahar-sauer-zaguia} If $B$ is a prime bichain, then $B$ and $B^t$ are the only bichains such that $o(B)=o(B^t)$.
\end{corollary}

\begin{lemma}\label{lem:zaguia98} Let $P$ be a poset of dimension two and $\{\leq_1,\leq_2\}$ be any realizer of $P$. The following propositions are true.
\begin{enumerate}[$(1)$]
  \item If $A$ is an interval of $P_1:= (V, \leq_1)$ and $P_2:= (V, \leq_2)$, then $A$ is a module of $P$.
  \item If $A$ is a totally ordered module of $P$, then $A$ is an interval of $P_1$ and $P_2$.
\end{enumerate}
\end{lemma}
\begin{proof}
\begin{enumerate}[$(1)$]
  \item The proof is easy and is left to the reader.
  \item Let $a<a'$ be two elements of $A$. Since $\leq_1$ and $\leq_2$ are linear extensions of $P$ we must have $a<_1 a'$ and $a<_2 a'$. We claim that the interval $[a,a']_{P_1}$ of $P_1$ is a  subset of the interval  $[a,a']_{P_2}$ of $P_2$. Indeed, let $x\in [a,a']_{P_1}$ and suppose for a contradiction that  $x\not \in [a,a']_{P_2}$. Suppose that  $x<_2 a$. Since $a<_1x$ then $x$ and $a$ are incomparable in $P$.  Since $A$ is a module of $P$, if $x\not \in A$  we must have $x$ incomparable to $a'$. But this is not the case: since  $a<_2 a'$ we have $x<_2a'$ and since $x<_1a'$ we have $x<a'$. Thus $x\in A$. But $A$ is a chain containing $a$ and therefore $x$ is comparable to $a$.
 We get a similar contradiction if $a' <_2 x$. Our claim is then proved. By symmetry we have $[a,a']_{P_2}\subseteq [a,a']_{P_1}$. Thus $[a,a']_{P_1}= [a,a']_{P_2}$. Now,  let $x\in [a,a']_{P_1}$. Since $[a,a']_{P_1}= [a,a']_{P_2}$, $x$ verifies $a<x<a'$. Since a module  is convex,  we must have $x\in A$. This shows that $[a,a']_{P_1}\subseteq A$. Hence, $A$ is an interval of $P_1$ and $P_2$ and we are done.
      \end{enumerate}
\end{proof}

We recall that if $P$ is a lexicographical sum $\sum_{i\in Q} Q_i$ and if $\{Q_{i,1}, Q_{i,2}\}$ is a realizer of $Q_i$ for each $i\in Q$ and $\{Q_1, Q_2\}$ is a realizer of $Q$ then $\{\sum_{i\in Q_1} Q_{i, 1}, \sum_{i\in Q_2} Q_{i, 2} \}$ is a realizer of $P$.

In general,  not every  realizer of $P$ is obtainable by this process.  A realizer of  $P$ induces  a realizer of each $Q_i$, but it does not necessarily induce a realizer on the quotient. A noticeable exception is when the $Q_i'$   are chains as we shall see below.

\begin{proposition}\label{prop:realizer}Let $P:= (V, \leq)$ be a poset   lexicographical sum $\sum_{i\in P'} C_i$ of chains $C_i:= (V_i, \leq_{\restriction V_i} )$ indexed by a poset $P':= (V', \leq')$ and let $\leq_1, \leq_2$ be two linear orders on $V$.
Then $\{\leq_1, \leq_2\}$ is a realizer of $P$ if and only if:
\begin{enumerate}[$(1)$]
 \item   $\leq_{1_{\restriction C_i}}$ and $\leq_{2_{\restriction C_i}}$ coincide with $\leq_{\restriction V_i}$  on  $C_i$;
 \item  $P_1:= (V, \leq_1)$ and $P_2:= (V, \leq_2)$ are respectively the lexicographical sums $\sum_{i \in P'_1} C_i$ and  $\sum_{i \in P'_2} C_i$,
 where $P'_1:= (V', \leq'_1)$,  $P'_2:= (V', \leq'_1)$ and  $\{\leq'_1,  \leq'_2\}$ is a realizer of $P'$.
 \end{enumerate}
 \end{proposition}
\begin{proof}
Suppose that $\{\leq_1, \leq_2\}$ is a realizer  of $\leq$.  Item 1 is immediate. From  Lemma \ref{lem:zaguia98}, each $V_i$ is an interval of $P_1$ and of $P_2$. Hence, $P_1$ and $P_2$ are the lexicographical sums stated in Item 2 and  $\{\leq'_1,  \leq'_2\}$ is a realizer of $P'$. The converse is immediate.
\end{proof}
\noindent {\bf Proof of Theorem \ref{thm:elzahar-sauer}}
Suppose that  at most one connected component of $\ainc(P)$, say $Q$,   is nontrivial and the only proper modules of the poset induced on this component are totally ordered. We prove that $P$ has a unique realizer. Note that  $P$ is the lexicographical sum $Q^{-}+Q+ Q^{+}$ where $Q^{-}$ and $Q^+$ are the set of lower bounds and  the set of upper bounds of $Q$ and furthermore $Q^{-}$ an $Q^{+}$ are chains.
 If $\{\leq_1, \leq_2\}$ is a realizer  of $\leq$,  then trivially  it induces a  realizer of $Q$.  Since the only proper modules of $Q$ are totally ordered, $Q$ is a lexicographical sum $\sum_{i\in Q'} C_i$ of chains $C_i:= (V_i, \leq_{\restriction V_i} )$ indexed by a poset $Q':= (V', \leq')$. According to Proposition \ref{prop:realizer}, the realizer on $Q$  induces a realizer on $Q'$. Since $Q'$ is prime, it follows from Corollary \ref{cor:elzahar-sauer} that this  realizer is unique. Hence,  $\{\leq_1, \leq_2\}$ is the unique realizer of  $P$.

Conversely, suppose that $P$ has a unique realizer.
First, $\ainc(P)$ cannot have two nontrivial components. Otherwise, let $Q$ and $Q'$ be such components;  let $Q_1,Q_2$ be a realizer of $Q$ and $Q'_1, Q'_2$ be a realizer of $Q'$.  With the realizers of the other components, we may built  a realizer of $P$ such that   $Q_1$ and $Q'1$ appear in one order and $Q_2$ and $Q'_2$ appear in the other order; switching $Q'_1$ and $Q'_2$ we get an other realizer. So $\ainc(P)$ has just one nontrivial component. We may suppose that $P$ reduces to that component, i.e., $\ainc(P)$ is connected. Now, suppose  for a contradiction that $P$ has a nontrivial module $A$ which is not a chain. Since $P$ has dimension two,  every induced subposet has dimension at most two, in particular $P_{\restriction A}$  and the quotient $Q:= P/A$ (obtained by identifying all elements of $A$ to a single element) have dimension at most two and in fact dimension two (since $A$ is not a chain and $\ainc(P)$ is connected). Let  $\{\leq_1,\leq_2\}$ be a realizer of $P_{\restriction A}$  and $\{L_{1}, L_{2}\}$ be a realizer of $P/A$. These realizers yield  two distinct  realizers of $P$.   This proves that $P$ is not uniquely realizable. A contradiction. The proof of the theorem is now complete.
\hfill $\Box$

\subsection{Proof of Theorem \ref {thm:bqo-poset-bichain}}\label{subsection:proofthm1}   Let $\mathcal{D}$ be a class of posets $P$ at  most countable whose order is the intersection of two linear orders,  $\ainc(P)$, the incomparability graph of $P$, is connected and the only nontrivial modules of $P$ are chains. Let $\mathcal C:= \sum \mathcal D$ be the collection of lexicographical sums over countable chains of members of $\mathcal{D}$.  Let  $\mathcal {C_{<\omega}}$ be the subclass of $\mathcal C$ made of finite posets.

\begin{proof}
$(1)$  The map associating to every bichain $B:= (V, \leq_1, \leq_2)\in o^{-1}(\mathcal C)$ the poset $(V, \leq_1\cap \leq _2)$ is a surjection of  $o^{-1}(\mathcal C)$ onto $\mathcal C$. It preserves the embeddability relation. Thus if $o^{-1}(\mathcal C)$ is bqo resp., wqo, its image $\mathcal C$ is bqo resp., wqo. Conversely,  suppose that $\mathcal{C}$ is bqo. Then $\mathcal D$  is bqo. It follows from Theorem \ref{thm:embed-bichain-poset}  that embeddings between elements of $\mathcal D$ yield embeddings (up to a transposition) between the corresponding elements of $o^{-1}(\mathcal{D})$, and hence $o^{-1} (\mathcal D)$ is bqo. It follows from Lemma \ref{lem:folklore} that every element of $o^{-1}(\mathcal{C})$ is a linear sum of elements of $o^{-1} (\mathcal D)$. From Laver's Theorem \ref{thm:laver} we deduce that $o^{-1}(\mathcal{C})$  is bqo.

$(2)$ Let  $\mathcal {C_{<\omega}}$ be the subclass of $\mathcal C$ made of finite posets.
The map from  $\mathcal {C_{<\omega}}$ to $\ainc\langle \mathcal C_{<\omega} \rangle$, associating to each poset $P:= (V, \leq )$ its incomparability graph $\ainc(P)$, preserves the embeddability, thus if ${\mathcal C_{<\omega}}$ is wqo,  resp.,  bqo, its image $\ainc \langle \mathcal C_{<\omega} \rangle$ is  wqo, resp.,  bqo. For the converse,  suppose that $\ainc \langle \mathcal{C_{<\omega}} \rangle$ is wqo,  resp., bqo. Since the members of  $\ainc \langle \mathcal D_{<\omega} \rangle$ are finite, it follows from Theorem \ref{thm:embed-poset-graph}  that embeddings between members of   $\ainc \langle \mathcal D_{<\omega} \rangle$  yield embeddings (up to duality) between the members  of $\mathcal{D_{<\omega}}$, and hence $\mathcal D_{<\omega}$ is wqo, resp., bqo.  Since  every element of $\mathcal{C}$ is a linear sum of elements of $\mathcal D$, from the fact that $\mathcal D_{<\omega}$ is wqo we deduce from Higman's Theorem that   $\mathcal {C_{<\omega}}$  is wqo  and from  the fact that $\mathcal D_{<\omega}$ is bqo we deduce  that $\mathcal C$ is bqo from the extension of Higman's theorem to bqo.
\end{proof}

\section{Minimal prime classes,   minimal prime structures and a proof of Theorems \ref{thm:minimalprimegraph} and \ref{thm:5}}\label{section:minimality}

Let $I$ be a set and $\Omega({I})$, resp., $\Omega(I)_{<\omega}$ be the class of binary structures, resp., finite binary structures  $R:= (V, (\rho_i)_{i\in I})$.  Let $\mathcal C\subseteq \Omega({I})$; we denote by $\prim (\mathcal C)$ the collection of prime structures belonging to $\mathcal C$.

There are interesting  hereditary classes $\mathcal C$ in  which $\prim (\mathcal C)$ is cofinal in $\mathcal C$, that is to say $\downarrow\prim  (\mathcal C)= \mathcal C$.

For an example, a consequence of Ille's result (Theorem \ref{thm:ille}) yields:
\begin{corollary}\label{cor:ille} If $\mathcal C$ is the age of a binary structure then there is some prime binary structure $R$ such that $\age (R)= \mathcal C$ if and only if $\downarrow \prim  (\mathcal C)= \mathcal C$.
\end{corollary}

A hereditary subclass
$\mathcal{C}$ of $\Omega(I)_{<\omega}$ is \emph{minimal prime} if it contains infinitely many prime structures but every proper hereditary subclass  of $\mathcal{C}$ contains only finitely many.  This later condition means that  $\mathcal C= \downarrow \mathcal I$ for every infinite subset $\mathcal I $ of prime members of $\mathcal C$.

A related notion is this: A  binary relational structure  $R$ is \emph{minimal prime} if $R$ is prime and every prime induced structure  with the same cardinality as $R$ embeds a copy of $R$.

We described the countable minimal prime graphs with no infinite clique (Theorem 2 \cite{pouzet-zaguia2009}, p.358). From  our description there exists four countable  prime graphs such that every infinite prime graph with no infinite clique embeds one of these graphs. All these four graphs are bipartite and only two are incomparability graphs. These are $P_\infty$ and $H$. In fact, $P_\infty$ is the incomparability graph of $P_2$ and $H$ is the incomparability graph of $P_1$. Since a poset $P$ is minimal prime if and only if $comp(P)$ is minimal prime (Proposition 4 of \cite{pouzet-zaguia2009}) it follows that every infinite prime poset with no infinite antichain embeds one of the posets $P_1,P_1^*,P_2,P_2^*$.
 Theorem \ref{thm:minimalprime} asserts that the  ages of $P_1$, $P_2$, $\comp (P_1)$  and $\comp(P_2)$ are  minimal prime. The question to know wether or not the age of a relational structure is minimal prime whenever $R$ is minimal prime was considered by Oudrar (\cite{oudrar} Probl\`eme 1, page 99). Her and our  results suggest that the answer is positive. But, note that a description of minimal prime bichains or posets  is not known; the same question for minimal prime ages is unsolved too.

We recall two results obtained by Oudrar and the first author (see Theorem 5.12, p. 92, and Theorem 5.15, p. 94 of  \cite{oudrar}).
\begin{theorem}\label{minimal}
Every   hereditary subclass of   $\Omega (I)_{<\omega}$  containing only finitely many members of size $1$ or $2$ and infinitely many prime structures   contains a minimal prime hereditary subclass.
\end{theorem}

\begin{theorem} \label{thm:main1} Every  minimal prime hereditary class is the age of some prime structure; furthermore this age is wqo.
\end{theorem}

We prove  the following result.

\begin{theorem}\label{thm:minimalprime}
\begin{enumerate}[$(1)$]
\item Let $P:=(V,\leq)$ be a  poset.  Then  $\age(\ainc (P))$ is minimal prime if and only if
$\age(\comp(P))$ is minimal prime. Furthermore,  $\age(P)$ is minimal prime if and only if $\age (\ainc(P))$ is minimal prime and $\downarrow  \prim (\age(P))= \age (P)$.
\item  Let $B:=(V,(\leq_1, \leq_2))$ be a bichain and $o(B):= (V, \leq_1\cap \leq_2)$. Then  $\age (B)$ is minimal prime if and only if $\age (o(B))$ is minimal prime and $\downarrow  \prim (\age (B))= \age (B)$.
\end{enumerate}
\end{theorem}

This  result represents two instances of a more general situation.

Let $I$ and $J$ be two finite sets. Let $\mathcal C$ be a hereditary subclass of $\Omega({I})$. A map $\varphi$  from  $\mathcal C$ to $\Omega({J})$ is a \emph{free operator} if it associates  to every binary structure $R:= (V, (\rho_i)_{i\in I})$ a binary structure $\varphi (R):=(V, (\tau_j)_{j\in J})$ in such a way that  for every $R, R'$ in $\mathcal C$,  every local isomorphism from $R$ to $R'$ induces a local isomorphism from $\varphi (R)$ to $\varphi (R')$. We will denote by $\varphi \langle \mathcal I \rangle$ the range by $\varphi$ of any subset $\mathcal I$ of $\mathcal C$.
In the following lemma, we suppose that $\varphi$ is such that $\varphi (R)$ is prime for each $R$ prime belonging to $\mathcal C$. Furthermore, we suppose that there is a free and involutive operator  on $\mathcal C$, denoted by $\dual$,   associating to each  $R\in \mathcal C$  some element  $\dual (R)$ of $\mathcal C$ in such a way that first $R$ is prime if and only if $\dual(R)$ is prime and next that for every $R, R'\in \mathcal C)$, $R$ prime and $\varphi(R)=\varphi(R')$ imply $R'=R$ or $R'= \dual (R)$. Under these conditions we have:

\begin{lemma}\label{lem:iteration}Let $R\in  \mathcal C$. Then:
\begin{enumerate}[$(1)$]
\item $\varphi \langle \prim (\age (R)) \rangle=\prim (\age (\varphi (R))$.
\item If  $\downarrow  \prim (\age(R))= \age (R)$ then for every $\mathcal I\subseteq \prim (\age (R))$, the range of $\mathcal I$,
$\downarrow \mathcal I= \age (R)$ if and only if $\downarrow{\varphi} \langle \mathcal I \rangle=\age (\varphi(R))$. In particular,
$\downarrow{\prim (\age (\varphi (R))}=\age (\varphi(R))$.
\item $\age (R)$ is minimal prime if and only if $\age (\varphi (R))$ is minimal prime and $\downarrow  \prim (\age(R))= \age (R)$.
\end{enumerate}
\end{lemma}
\begin{proof}

We start with the following observations
\begin{enumerate}[{(a)}]
\item For every $R\in \mathcal C$ and every subset $F$ of the domain of $R$, $\varphi(R_{\restriction F})= \varphi (R)_{\restriction F}$.
\item  $R\in \prim  (\mathcal C)$ if and only if $\varphi(R)\in \prim (\varphi \langle \mathcal C \rangle)$.

\item If $\varphi(Q)\leq \varphi (R)$ and $Q$ is prime then   $Q\leq R$ or $Q\leq \dual(R)$.
\end{enumerate}

Property $(a)$ is valid for every free operator. Concerning property $(b)$ note that if $A$ is a module of $R$ this is also a module of $\varphi (R)$, hence if $\varphi(R)\in \prim (\varphi \langle \mathcal C \rangle)$ then $R\in \prim  (\mathcal C)$. The converse holds by our hypothesis on $\varphi$.
We prove $(c)$.
Suppose $\varphi(Q)\leq \varphi (R)$.   Let $f$ be an embedding  from $\varphi(Q)$ to $\varphi (R)$. Let $Q'$ be the inverse image of $R_{\restriction(range(f)}$ by $f$. Since $f$ is an embedding  of $Q'$ in $R$,  $Q'\leq R$, and since $\varphi$ is a free operator, $f$ is an embedding from $\varphi (Q')$ in $\varphi(R)$; with the fact that $f$ is an embedding from $\varphi(Q)$ to $\varphi (R)$ we have   $\varphi (Q')= \varphi (Q)$. Since $QS$ is prime,  either $Q'=Q$ or $Q'= \dual(Q)$. In the first case $Q\leq R$, while in the second case,  $\dual (Q) \leq R$ hence $Q= \dual( \dual (Q))\leq \dual (R)$.

Item $(1)$. First, $\varphi \langle \prim (\age (R)) \rangle \subseteq \prim (\age (\varphi (R))$. Let $Q\in \prim (\age (R))$. Due to the condition on $\varphi$, $\varphi (Q)$ is prime, hence $\varphi (Q)\in \prim(\age (\varphi (R))$. Conversely, $ \prim (\age (\varphi (R))\subseteq \varphi \langle \prim (\age (R)) \rangle$. Let $S\in \prim (\age (\varphi (R))$. Then $S$ is isomorphic to $\varphi (R)_{\restriction F}$ for some finite subset of the domain of $F$. According to $(a)$,  $\varphi (R)_{\restriction F}= \varphi (R_{\restriction F} $ and according to $(b)$, $R_{\restriction F} $ is prime, hence $S\in  \varphi \langle \prim (\age (R)) \rangle$.

Item $(2)$. Suppose first that  $\downarrow \mathcal I= \age (R)$. The fact that $\downarrow{\varphi}\langle \mathcal I \rangle=\age (\varphi(R))$ is immediate. Indeed, let $T\in \age  (\varphi(R))$. Then, up to isomorphy, $T= \varphi (R)_{\restriction F}$ for some finite subset of the domain of $R$, hence $T= \varphi (R_{\restriction F})$. Since $\downarrow \mathcal I= \age (R)$, there is some $S\in \mathcal I$ such that $R_{\restriction F}\leq S$. Thus, $T\leq \varphi(S)\in \varphi\langle \mathcal I \rangle$, hence $T\in  \downarrow{\varphi}\langle \mathcal I \rangle$ as required.
Suppose next that $\downarrow{\varphi}\langle \mathcal I \rangle=\age (\varphi(R))$. We prove that $\downarrow \mathcal I= \age (R)$. We suppose that this is not the case. Let $Q_{0}\in \age (R)\setminus  \downarrow \mathcal I$. Since $\downarrow  \prim (\age(R))= \age (R)$,  we may suppose that $Q_{0}$ is prime.  Since $\downarrow{\varphi}\langle \mathcal I \rangle=\age (\varphi(R))$ there exists  $R_{0}\in   I$ such that $\varphi(Q_{0})\leq \varphi (R_{0})$. From Property (c) above,  $Q_{0}$ embeds in  $R_{0}$ or in $\dual(R_{0})$. Since $Q_{0}\not \in \downarrow \mathcal I$, and $R_{0}\in \downarrow \mathcal I$, $Q_{0}$ does not embed in $R_{0}$. But then $Q_{0}$ embeds into $\dual (R_{0})$, that is $\dual (Q_{0}) \leq R_{0}$. Since $R_0\in \downarrow \mathcal I$, $\dual (Q_{0})\in \downarrow \mathcal I$, hence $\dual (Q_{0})\in \age (R)$. Since $\age (R)$ is up-directed, there is some $Q_1 \in \age (R)$ which embeds $Q_0$ and  $\dual (Q_0)$. Since $\dual (Q_{0})\leq Q_1$, $Q_{0}\leq \dual (Q_1)$. Thus, if we prove that $\dual (Q_1)\in \downarrow \mathcal I$ we will get a contradiction. We may  suppose $Q_1$ prime. As above, there is some $R_1\in \mathcal  I$  such that  $\varphi(Q_{1})\leq \varphi (R_{1})$. Thus with Property $(c)$, $Q_{1}$ embeds into $R_{1}$ or $\dual (R_{1})$. By the same argument as above, $\dual (Q_{1})\leq R_{1}$. But then $\dual (Q_{1})\in \downarrow \mathcal I$. This yields the desired contradiction.

\noindent $(3)$ Suppose that $\age(R)$ is minimal prime. Then trivially, $\downarrow \prim (\age (\varphi (R)))= \age (\varphi (R))$. We prove that $\age (\varphi (R))$ is  minimal prime. Since $\age(R)$ is minimal prime it contains infinitely many prime elements. Since the images by $\varphi$ of primes are primes,  $\age (\varphi (R))$ contains infinitely many prime elements. Let $\mathcal I$ be an infinite  set of prime elements of $\age (\varphi (R))$. Since the inverse image of a prime is prime, $\varphi^{-1}(\mathcal I)$ is made of prime elements of $\age(R)$.  Let  $\downarrow \varphi^{-1}(\mathcal I)$ be  the initial segment of $\age(R)$ generated by $\downarrow \varphi^{-1}(\mathcal I)$.  Since $\age(R)$  is minimal prime,  $\downarrow \varphi^{-1}(\mathcal I)$ equals $\age (R)$.  Applying $\varphi$, we get $\varphi \langle \downarrow \varphi^{-1}(\mathcal I) \rangle \subseteq \downarrow \mathcal I$, thus $\age (\varphi(R))\subseteq \downarrow \mathcal I$. It follows that  $\age (\varphi(R))$ is minimal prime. \\
Suppose $\age(\varphi(R))$ is minimal prime and $\downarrow  \prim (\age(R))= \age (R)$.
Let $\mathcal I\subseteq \prim (\age(R))$.  Our aim is to show that $\downarrow \mathcal I=\age (R)$ if $\mathcal I$ is infinite. Apply $(2)$ of this lemma. If $\mathcal I$ is infinite then $\downarrow \varphi \langle \mathcal I \rangle$ is infinite and included into $\prim (\age (\varphi(R)))$. Since that age is minimal prime,  $\downarrow \varphi \langle \mathcal I \rangle =\age \varphi (R)$.  Then  $(2)$
yields   $\downarrow \mathcal I=\age (R)$ as required. \end{proof}

\noindent{\bf Proof of Theorem \ref{thm:minimalprime}.}
\noindent $(1)$ The fact that $\age (\comp (P))$ is minimal prime if and only if $\age (\ainc(P))$ is minimal prime is obvious and requires no condition on $P$. In fact, since a graph $G$ and its complement  $G^c$ have the same local isomorphisms, it is a simple exercise to check that  $\age (G)$  is minimal prime if and only if  $\age (G^c)$ is minimal prime. For the other part of the theorem  take for $\mathcal C$ the collection of finite posets, for $\varphi$ the map associating to each finite poset  $P$ its incomparability graph $\ainc(P)$ and for the involution $\dual$ on $\mathcal C$ the ordinary duality   associating to each poset the dual poset.  Theorem \ref{kelly} expresses that the condition on the involution is satisfied. Item $3$ of Lemma \ref{lem:iteration} yields  $(1)$. \\
\noindent $(2)$ Take for $\mathcal C$ the collection of finite bichains, for $\varphi$ the map associating to each finite bichain $B:= (V, (\leq_1, \leq_2))$ the poset $o(B):= (V, \leq_1\cap \leq_2))$  and for the involution $\dual$ on $\mathcal C$ the transpose,  associating $B^t:= (V, (\leq_2, \leq_1))$ to $B:= (V, (\leq_1, \leq_2))$. Theorem \ref{zaguia} and Corollary \ref{cor:elzahar-sauer-zaguia} express that the condition on the involution is satisfied. Item $3$ of Lemma \ref{lem:iteration} is $(2)$.

\subsection {Proof of Theorem \ref{thm:minimalprimegraph}.}

\begin{claim} \label{claim:minimalprime}A hereditary class $\mathcal C$ of finite bipartite permutation graphs  is minimal prime if and only if $\mathcal C$ is $\age (H)$ or $\age(P_{\infty})$.
\end{claim}
\noindent {\bf Proof of  Claim \ref{claim:minimalprime}.}
Suppose that $\mathcal C$ is minimal prime. Then $\mathcal C= \age (G)$  where $G$ is a prime graph which is  the incomparability graph of a poset of width two (Theorem  \ref{thm:main1}). From Theorem 2 of \cite{pouzet-zaguia2009} follows that $G$ embeds $P_\infty$ or $H$. Hence, $\mathcal{C}$ contains  $\age(P_\infty)$ or $\age(H)$. Since $P_\infty$ and $H$ are prime and $\mathcal{C}$ minimal prime,  $\mathcal{C}= \age(P_\infty)$ or $\age(H)$.\\
Conversely, let us  prove  that $\age(P_\infty)$ and $\age(H)$ are minimal prime. This is obvious for $\age(P_\infty)$, indeed, finite prime subgraphs  of $P_{\infty}$ on  at least three vertices are paths, thus form a chain for the embeddability order; if a  hereditary subclass of  $\age (P_{\infty})$ contains infinitely many  of these paths it contains all, thus is equal to $\age (P_{\infty})$. For $\age (H)$, the situation is about the same: finite prime induced subgraphs of $H$ on at least four vertices  form also a chain. Indeed, these subgraphs are the  \emph{critically prime} discovered by  Schmerl and Trotter \cite{S-T}, prime graphs such that the deletion of any vertex yields a nonprimitive graph. As shown by Schmerl and Trotter, these graphs have an even number of vertices;   up to complementation, there is just one critically prime graph on an even number of vertices.
\hfill $\Box$

With Claim \ref{claim:minimalprime} and  Theorem \ref{thm:minimalprime} we obtain:
\begin{claim} \label{claim:minimalprime2} If  $\mathcal C$ is a hereditary class of finite posets of width two then $\mathcal C$ is minimal prime if and only if $\mathcal C$ is $\age (P_1)$ or  $\age (P_2)$. \end{claim}
 Indeed, if $\mathcal C$ is minimal prime then $\mathcal C= \age (P)$ for some prime poset $P$ of width at most two.  Conversely, if a hereditary class $\mathcal D$ of incomparability graphs of finite posets of width two is minimal prime, then $\mathcal D= \age (G)$ for some prime graph and  this graph is of the form $\ainc(P)$ for some prime poset $P$ of width at most two.   According to Theorem \ref{thm:minimalprime}, $\mathcal C$ is minimal prime if and only if $\mathcal D$ is minimal prime. Since  $\age (H)$ and $\age (P_{\infty})$ are the only minimal prime classes of incomparability graphs of posets of width two then $\age (P_1)$ and $\age(P_2)$ are the  only minimal prime classes of posets of width two.  \hfill $\Box$

\subsection{Proof of Theorem \ref{thm:5}.} For the proof, note first that  $\age (B_1)\not = \age(B_1^t)$ while  $\age (B_2)\not = \age(B_2^t)$.  Now, let $\mathcal C$ be a minimal prime class of finite bichains.  Then $\mathcal C$ is the age of some bichain $B$ that we may suppose prime. By $(2)$ of Theorem \ref{thm:minimalprime},  $\age( o(B))$  is minimal prime. If $o(B)$ has width at most two, then by
Theorem \ref{thm:minimalprimegraph},  $\age (o(B))$ is either  $\age (P_1)$ or $\age (P_2)$.  Since $B$ is prime, then by Corollary \ref{cor:elzahar-sauer-zaguia},  $B$ and  $B^t$ are the only bichain such that $o(B)= o(B^t)$. Hence, there are at most   two ages, namely $\age (B_1)$, and  $\age (B_1^t)$ such that $\age (o(B_1))=\age (o(B_1^t)=\age (P_1)$ and at most two ages, $\age (B_2)$, and  $\age (B_2^t)$ such that $\age (o(B_2))=\age (o(B_2^t)=\age (P_2)$. Since $\age (B_2)\not = \age(B_2^t)$, this provides only three ages. Now since,    $\mathcal C$ is minimal prime and $\age (B)\subseteq  \mathcal C$, $\age (B)= \mathcal C$.  Thus  $\mathcal C$  is  the age of one of the  bichains $B_1$, $B_1^t$ or $B_2$ as claimed.
This completes the proof of the theorem.
\hfill $\Box$

\begin{comments}\label{comments-uniformly}
a) The involutive and free operator   $\dual$ does not appear in the conclusions of Lemma \ref{lem:iteration}. This suggests that this lemma holds   under more general conditions. However, some conditions are necessary.
Take for $\mathcal C$ the collection of finite oriented graphs $G$   such that the symmetric hull of $G$ is a direct sum of  finite paths. Let $\varphi$ be the map associating to each graph its symmetric hull. The class $\mathcal C$ is a hereditary class and $\varphi $ is a  free operator. Members of $\mathcal C$ are prime if and only if their symmetric hull is  prime, indeed their  symmetric hull must have  at most two vertices or  must be  a path. Thus the first conditions of the frame  of our lemma are satisfied. The class $\mathcal C$ is an age and the age of a prime structure. It is not minimal prime (it contains the age of an infinite zig-zag) but the range  $\varphi \langle \mathcal C \rangle$ of $\mathcal C$ being  the age of the infinite path $P_{\infty}$, it is  minimal prime. Hence $(3)$ of Lemma \ref{lem:iteration} does not hold. In fact, $\mathcal C$ contains $2^{\aleph_0}$ minimal prime ages, each one transformed to $\age(P_{\infty})$ by $\varphi$. For a  proof of this fact, code orientations of $P_{\infty}$ by  words $u$ with domain $\NN$ on the two letters alphabet $\{0, 1\}$. To a $0$-$1$ word $u:= (u_n)_{n\in \NN}$ on $\NN$ associate an orientation of $P_{\infty}$, orienting the edge $\{n, n+1\}$ from  $n$ to $n+1$ if $u_n=1$ and from $n+1$ from $n$ if $u_n=0$. Then observe  that the age of the orientation is minimal prime if and only if the word $u$ is  uniformly recurrent (a word $u$ on $\NN$  is \emph{uniformly recurrent} if for every $n\in \NN$ there exists $m\in \NN$ such that each factor $u(p),...,u(p+n)$ of length $n$ occurs as a factor of every factor of length $m$). To conclude, use the fact that  there are $2^{\aleph_0}$ uniformly recurrent  words $u_\alpha$ on the two letters alphabet $\{0, 1\}$ such that for $\alpha \neq \beta$ the collections $Fac(u_\alpha)$ and $Fac (u_\beta)$ of their finite factors are distinct, and  in fact incomparable with respect to set inclusion (this is a well known fact of symbolic dynamic, e.g. Sturmian words with different slopes will do \cite{pytheas}). For more on the use of recurrent words,  see \cite{pouzet-zaguia} and \cite{oudrar-pouzet-zaguia}.

b) We do not know if in Theorem \ref{thm:minimalprime}  we can remove the condition $\downarrow  \prim (\age (P))= \age (P)$ in $(1)$ and the condition $\downarrow  \prim (\age (B))= \age (B)$ in $(2)$.
\end{comments}

\end{document}